\documentclass[a4paper,11pt,twoside]{article}

\usepackage{amsmath,amssymb, mathtools} 
\usepackage{amsthm} 
\usepackage{dsfont} 
\usepackage{tensor} 

\usepackage[pdftex]{graphicx} 
\usepackage{color}
\usepackage[colorlinks,citecolor=blue,linkcolor=blue,urlcolor=blue]{hyperref} 
\usepackage{cleveref}
\usepackage{authblk} 

\usepackage[utf8]{inputenc} 
\usepackage[english]{babel} 

\usepackage{geometry} 
\usepackage{tikz} 
\usepackage{cite} 
\usepackage{enumitem} 

\newtheorem{thm}{Theorem}
\newtheorem{conj}[thm]{Conjecture}
\newtheorem{lemma}[thm]{Lemma}
\newtheorem{coro}[thm]{Corollary}
\newtheorem{definition}[thm]{Definition}
\newtheorem{prop}[thm]{Proposition}
\newtheorem{example}[thm]{Example}
\newtheorem{Remark}[thm]{Remark}

\newtheorem{fakethm}{Theorem}
\newtheorem{fakecoro}[fakethm]{Corollary}

\numberwithin{thm}{section} 
\numberwithin{equation}{section} 
\numberwithin{figure}{section} 

\newcommand*{\R}{\mathbb{R}}
\newcommand*{\C}{\mathbb{C}}
\newcommand*{\Z}{\mathbb{Z}}
\renewcommand*{\S}{\Sigma}

\newcommand*{\He}{\mathcal{H}}
\renewcommand*{\l}{\lambda}
\newcommand*{\mf}{\mathfrak}
\newcommand*{\mc}{\mathcal}
\newcommand*{\bb}{\mathbb}

\DeclareMathOperator{\tr}{tr}

\DeclareMathOperator{\Span}{Span}

\DeclareMathOperator{\id}{id}

\DeclareMathOperator{\Hom}{Hom}

\DeclareMathOperator{\End}{End}

\DeclareMathOperator{\SL}{SL}
\DeclareMathOperator{\GL}{GL}

\DeclareMathOperator{\ram}{ram}
\DeclareMathOperator{\Cyc}{Cyc}
\DeclareMathOperator{\Irr}{Irr}

\title{Topological quantum field theories from Hecke algebras}

\author[1]{Vladimir Fock \thanks{fock@math.unistra.fr}}
\author[1]{Valdo Tatitscheff \thanks{valdo.tatitscheff@normalesup.org}}
\author[2]{Alexander Thomas \thanks{a.thomas@mpim-bonn.mpg.de}}
\affil[1]{\footnotesize{IRMA, UMR 7501, Universit\'e de Strasbourg et CNRS \\ 
		7 rue Ren\'e Descartes 67000 Strasbourg, France}}
\affil[2]{\footnotesize{Max-Planck-Institut für Mathematik Bonn, Vivatsgasse 7, 53111 Bonn, Germany}}

\date{}

\begin{document}
	
	\maketitle

	\begin{abstract}
		We construct two-dimensional non-commutative topological quantum field theories (TQFTs), one for each Hecke algebra corresponding to a finite Coxeter system. These TQFTs associate an invariant to each ciliated surface, which is a Laurent polynomial for punctured surfaces. There is a graphical way to compute the invariant using minimal colored graphs. We give explicit formulas in terms of the Schur elements of the Hecke algebra and prove positivity properties for the invariants when the Coxeter group is of classical type, or one of the exceptional types $H_3$, $E_6$ and $E_7$.
	\end{abstract}

	\section{Introduction}
	
	Iwahori--Hecke algebras (referred to as \textit{Hecke algebras} in the sequel) are remarkable associative non-commutative deformations of Coxeter groups depending on a parameter $q$. We use these Hecke algebras to construct topological invariants of ciliated surfaces. This construction behaves nicely under gluing along the boundary of the surfaces and hence defines a topological quantum field theory (TQFT).
	
	The origin of this construction comes from the study of character varieties of surface groups. The attempt is to generalize Thurston's laminations which can be identified with a basis of the function space of the $\mathrm{SL}_2(\C)$-character variety. It turns out that one is led to consider affine Hecke algebras. The question arose what happens if one uses a finite Hecke algebra. This led to the TQFT and the graphical calculus for Hecke algebras presented in this paper.
	
Two-dimensional TQFTs can be classified by algebraic objects. Closed TQFTs are in bijection with Frobenius algbras, and open-closed TQFTs with so-called Cardy--Frobenius algebras \cite{lauda2008open, alexeevski2007hurwitz}. The Hecke algebra is a particular example of a Cardy--Frobenius algebra. We see the main contribution of our paper in the natural way the TQFT arises, in the diagrammatic calculus and in the positivity properties of the constructed TQFT. 

	\bigskip
	Let $G$ be a finite-dimensional simple Lie group over $\mathbb{F}_q$, where $q=p^\alpha$ is the power of a prime number. Let $H$ be a Cartan subgroup, $B$ a corresponding Borel subgroup and $W\simeq \mathrm{Norm}(H)/H$ the Weyl group. The Hecke algebra $\mathcal{H}_G^q$ is the algebra of functions on $G$ invariant under left and right shift by $B$, or equivalently the algebra of functions of the double coset $B\backslash G/B$, with product given by the convolution divided by the order of $B$. Since as a set the Weyl group $W$ is in bijection with $B\backslash G/B$, the Hecke algebra $\mathcal{H}_G^q$ is a deformation of the group algebra $\mathbb{C}[W]$. 
	
	The structure constants in the corresponding Hecke algebra $\mathcal{H}_G^q$ are related to the counting of $\mathbb{F}_q$-points of flag varieties associated to the corresponding algebraic group (see e.g. \cite{curtis1988representations}). In a judiciously chosen basis the (modified) structure constants of $\mathcal{H}_G^q$ are Laurent polynomial in $q$ invariant under cyclic permutations of their indices. Subsequently, one can associate such a structure constant to an oriented topological triangle with oriented sides labeled with the corresponding basis elements of $\mathcal{H}_G^q$, without the need to fix an initial edge. This construction extends naturally to the Hecke algebras corresponding to finite Coxeter systems $(W,S)$ which are not Weyl groups.
	
	Now one can consider gluing two triangles along one of their edges if its orientation and labels coincide. Summing over all possible labels for the edge along which the gluing is performed, that is over all basis vectors of the Hecke algebra, yields a Laurent polynomial that only depends on the labels on the exterior edges. Repeating this procedure one obtains a way to associate a Laurent polynomial $P \in \Z[v^{\pm 1}]$ to any triangulated ciliated surface $\S$ whose boundary is labeled by elements of the Weyl group $W$:
	
	\begin{equation}
	    P_{\S,W}(v) = \sum \prod_f c_f(v) ~,
	\end{equation}
	where the sum runs over all labelings of the inner edges of the triangulation by elements of $W$, the product runs over all faces $f$, and $c_f(v) \in \Z[v^{\pm 1}]$ denotes the structure constant associated to a triangle $f$ (see \Cref{Sec:invpuncsur} for the precise definition).
	
	The definition of the polynomial $P$ as a state sum uses a triangulation of the surface. Associativity in the Hecke algebra gives (see \Cref{Thm:trianginv}):
	
	\begin{fakethm}
		The polynomial invariant is independent of the triangulation. Hence it is a topological invariant of the ciliated surface. 
	\end{fakethm}
	
	In \Cref{Sec:graph-calcul}, we introduce a diagrammatic way to compute the product in the Hecke algebra when $W$ is a Weyl group. From this we get a graphical calculus of our invariants using graphs with edges labeled by simple reflections in $W$. These graphs are called higher laminations. They emerge from the interpretation of the structure constants in the Hecke algebra in terms of configurations of triples of flags. \Cref{Fig:ex-higher-lam} gives two examples of graphs for $W=\mf{S}_3$.
	
	\begin{figure}[h!]
		\centering
		\includegraphics[height=4cm]{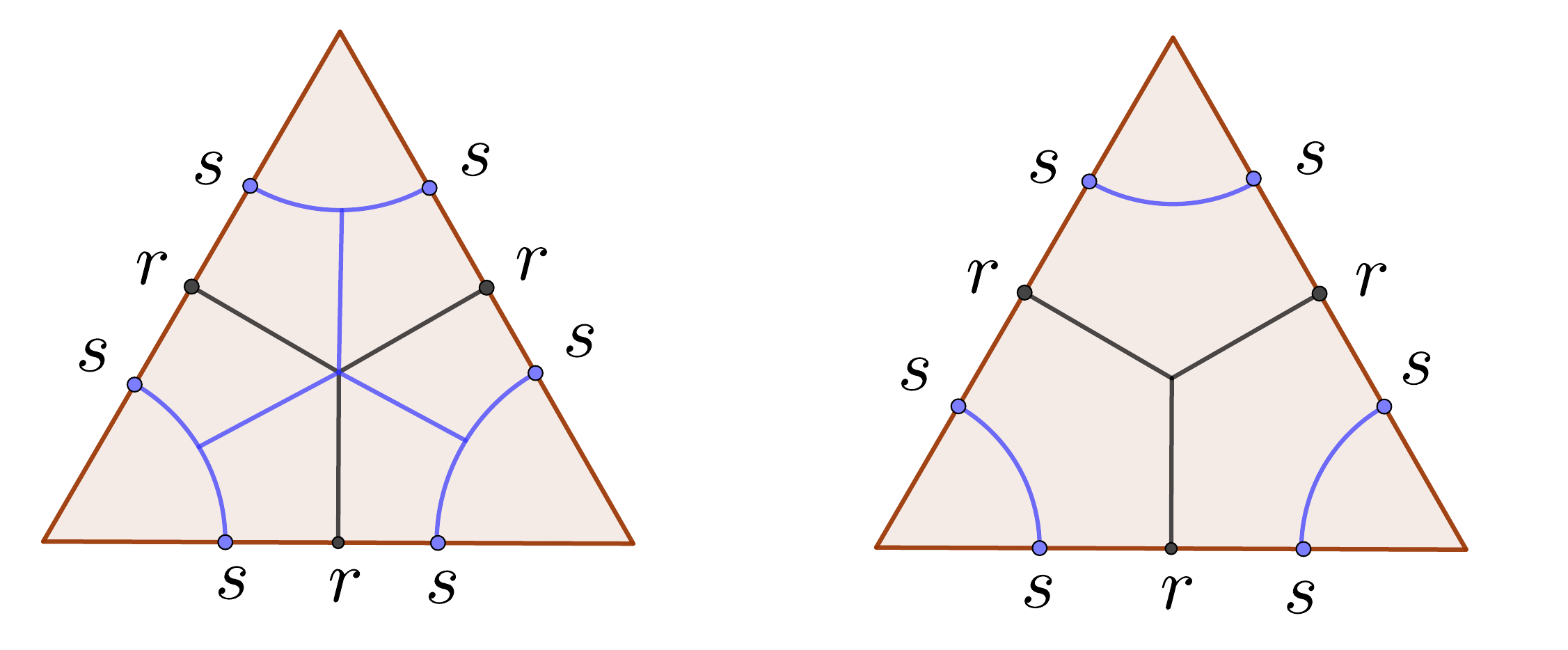}
		
		\caption{Examples of graphs for $W=\mf{S}_3$}\label{Fig:ex-higher-lam}
	\end{figure}
	
	Counting higher laminations with some weight gives a diagrammatic way to compute the polynomial invariant (see \Cref{Thm:poly-via-laminations}):
	\begin{fakethm}
		The polynomial invariant for a ciliated surface $\Sigma$ and a Weyl group $W$ is given by $$P_{\S,W}(Q) = \sum_\Gamma Q^{\ram(\Gamma)}$$
		where the sum runs over all higher laminations $\Gamma$ of type $W$ and $Q^{\ram(\Gamma)}$ is the weight associated to $\Gamma$.
	\end{fakethm}
	
	As a consequence we get a symmetry in the invariant that for closed surfaces (see \Cref{Prop:invariance-poly}):
	\begin{fakecoro}
		For a closed surface, the invariant is a polynomial in $q$, where $Q=q^{1/2}-q^{-1/2}$, invariant under the transformation $q\mapsto q^{-1}$.
	\end{fakecoro}

	In \Cref{Sec:TQFT}, we extend our construction to ciliated surfaces whose boundary is labeled by elements of the Hecke algebra. Our main result is then the following (see \Cref{Thm:Hecke-TQFT}):
	
	\begin{fakethm}
		Our construction satisfies the axioms of a $2$-dimensional\footnote{In the notation of Atiyah, we construct a one-dimensional TQFT. In modern language we speak about a $d$-dimensional TQFT for $d$ being the dimension of the bordisms, so 2 dimensions in our case.} TQFT as defined in \cite{atiyah1988topological}.
	\end{fakethm}
	
	Hence we construct a family of TQFTs, one for each finite Coxeter system. The interesting feature of these TQFTs is that they are non-commutative. This ultimately comes from the fact that the cobordisms we consider are ciliated surfaces, whose boundary is a disjoint union of segments joining adjacent cilia. 
	
	Using the gluing property of a TQFT, we can decompose ciliated surfaces into elementary parts (marked tori and marked cylinders). The center of the Hecke algebra plays a prominant role. Since the structure of the latter is well understood we can derive an explicit expression for the invariant in terms of the Schur elements of the Hecke algebra (see \Cref{Thm:explexpr}):
	
	\begin{fakethm}
	For a ciliated surface $\Sigma$ of genus $g$ with $k$ marked points and $n$ boundary components equipped with labels $(h_1,...,h_n) \in \He^n$, the polynomial invariant is:
		\begin{equation}
		P_{\S,W}(q) = \sum_\l (\dim V_\l)^{k} s_\l(q)^{2g-2+k+n}\chi_\l(h_1)\cdots \chi_\l(h_n)~
		\end{equation}
		where the sum is taken over all irreducible representations $V_\l$ of the Hecke algebra and where $s_\l$ are the corresponding Schur elements.
	\end{fakethm}

	From these explicit expressions and a thorough analysis of the Schur elements, carried out in \Cref{Appendix:Schur-positive}, we can derive positivity properties of the invariant (\Cref{Thm:positivity} and \Cref{Coro:A-positive}):
	\begin{fakethm}
		The polynomial invariant has positive coefficients for all Coxeter groups of classical type and in the exceptional types $H_3, E_6$ and $E_7$. For all other types, the polynomial can have negative coefficients.
		
		In type $A$ and for boundary labels with positive coefficients in the Kazhdan--Lusztig basis, the polynomial invariant has positive coefficients.
	\end{fakethm}

	\paragraph{Notations.}
	We write \textit{Hecke algebra} for an Iwahori--Hecke algebra for a finite Coxeter system. When we speak about graphs, we always mean fat graphs, i.e. graphs embedded in some surface (planar graphs for example). We say that a polynomial is \textit{positive} if all its coefficients are positive. Furthermore, we use the following notations:
	
	\begin{center} 
		\begin{tabular}{p{0.16\textwidth}p{0.8\textwidth}} 
			$\Sigma_{g,k}$ & ciliated surface (see \Cref{Sec:ciliated-surface}) \\
			$(W, S)$ & Coxeter system with simple reflections $S$ \\
			$\He_{(W,S)}, \He$ & Iwahori--Hecke algebra of $(W,S)$ \\
			$v, q, Q$ & formal parameters in the Hecke algebra linked by \\
			& $q=v^{-2}$ and $Q=q^{1/2}-q^{-1/2}$
		\end{tabular}
	\end{center}

	\paragraph{Acknowledgments.}
	We warmly thank Maria Chlouveraki for her various inputs about Schur elements, Sebastian Manecke for his contribution for the positivity, Malte Lackmann for showing us SageMathCell, Wille Liu and Daniel Tubbenhauer for discussions.
	A. Thomas acknowledges support from the Max-Planck-Institute for Mathematics in Bonn.

	\tableofcontents

	\section{Preliminaries}
	
	We recall the definition of the Hecke algebra of a Coxeter system and introduce its standard and Kazhdan--Lusztig bases following \cite{elias2016soergel,libedinsky2019gentle}, and then set our conventions and notations for ciliated surfaces as in section 2 of \cite{fock2007dual}.

	\subsection{Hecke algebras and their standard basis}\label{Sec:Hecstand}
	
	Let $(W,S)$ be a finite Coxeter system. Given two reflections $s,t\in S$ let $m_{st}\in\mathbb{N}\cup\{ \infty \}$ denote the order of $st$. With this notation\footnote{By definition of a Coxeter system, for all $s\in S$ one has $m_{ss} = 1$.}, $W$ can be presented as:
	\begin{equation}
	W = \left\langle s \in S~\vert~ (st)^{m_{st}} = \id \ \forall s,t \in S \right\rangle \ .
	\end{equation}
	Let $l:W\rightarrow\mathbb{N}_{\geq 0}$ and $\geq$ be respectively the length function and the Bruhat order on $W$.
	
	Let $v$ be an indeterminate. The \textbf{Hecke algebra} $\He_{(W,S)}$ corresponding to $(W,S)$ as defined in \cite{iwahori1964structure} is the associative $\mathbb{Z}[v^{\pm1}]$-algebra with generators $\{h_s\}_{s\in S}$ and relations of two types, the quadratic ones:
	\begin{equation}\label{Eq:Quadrel}
	h_s^2 = (v^{-1}-v)h_s+1
	\end{equation}
	for all $s\in S$, and the braid ones: 
	\begin{equation}\label{Eq:Braidrel}
	h_sh_r \dots = h_rh_s \dots 
	\end{equation}
	for all $s,t \in S$ such that $m_{st} < \infty$, in which case there are $m_{st}$ terms on each side. 
	
	Except when the explicit subscript $(W,S)$ is needed for clarity or in explicit examples, we will drop it in what follows and simply refer to the Hecke algebra corresponding to $(W,S)$ as $\He$.
	
	\begin{Remark}
		An equivalent, but different, way to describe the Hecke algebra is to use generators $(T_s)_{s\in S}$ which satisfy the same braid relations and where the quadratic relation is
		\begin{equation}\label{Eq:defq}
		T_s^2 = (q-1)T_s + q \ .
		\end{equation}
		for another indeterminate $q$.
		Apart from \Cref{Sec:counting}, we always use the the ``normalized'' version of the Hecke algebra where the quadratic relations take the form of \Cref{Eq:Quadrel} with $v=q^{-1/2}$. 
	\end{Remark}
	
	Let $\bar w = s_1 \dots s_k$ be a reduced expression for some $w\in W$ in terms of $s_1,\dots,s_k \in S$, and set:
	\begin{equation}
	h_{\bar w} = h_{s_1}\dots h_{s_k} \ .
	\end{equation}
	By the famous result of Matsumoto \cite{matsumoto1964generateurs} that every reduced expression for $w$ can be obtained from $\bar w$ using braid relations only, the element $h_{\bar w}$ does not depend on the choice of a reduced expression for $w$ and one can define $h_w := h_{\bar w}$. Let also $h_e := 1$. 
	
	\begin{lemma}
		The set $\{h_w\}_{w\in W}$ is a basis of $\He$ as a free $\mathbb{Z}[v^{\pm1}]$-module. It is called the standard basis of $\He$.
	\end{lemma}
	
	Let $s \in S$ and $w \in W$. The multiplication in $\He$ can be rewritten as:
	\begin{equation}
	h_sh_w = \left\{ 
	\begin{array}{lcc}
	h_{sw} & \mathrm{if} & w \leq sw \\
	(v^{-1}-v) h_w + h_{sw} & \mathrm{if} & sw \leq w
	\end{array}
	\right. \ .
	\end{equation}
	
	\begin{example}
		The Coxeter system of type $A_1$ is $(\mathfrak{S}_2,\{s\})$, where $\mathfrak{S}_2$ is the group of permutations of a set of two elements, and $s$ is its generating involution. The corresponding Hecke algebra $\He_{(\mathfrak{S}_2,\{s\})}$ has basis $(h_e,h_s)$ as $\mathbb{Z}[v^{\pm1}]$-module, and the relations $h_sh_e=h_eh_s$ and $h_s^2 = (v^{-1}-v)h_s+1$ hold.
	\end{example}
	
	\begin{example}
		The Coxeter system of type $A_2$ is $(\mathfrak{S}_3,\{s,t\})$, where $s$ and $t$ are two transpositions generating $\mathfrak{S}_3$. The standard basis of the Hecke algebra $\He_{(\mathfrak{S}_3,\{s,t\})}$ as a $\mathbb{Z}[v^{\pm1}]$-module is $(h_e, h_s, h_t, h_{st}, h_{ts}, h_{sts})$. However as a $\mathbb{Z}[v^{\pm1}]$-algebra, $\He_{(\mathfrak{S}_3,\{s,t\})}$ is generated by $h_e, h_s$ and $h_t$ only, where $h_e$ is in the center and with the quadratic relations of \Cref{Eq:Quadrel} for $h_s$ and $h_t$, and the braid relation:
		\begin{equation}
		h_sh_th_s = h_th_sh_t \ .
		\end{equation}
	\end{example}
	
	\begin{example}
		The Coxeter system of type $B_2$ is $(\mathbb{D}_4,\{s,t\})$, where $s$ and $t$ are two transpositions which generate the dihedral group $\mathbb{D}_4$. The standard basis of the corresponding Hecke algebra $\He_{(\mathbb{D}_4,\{s,t\})}$ is $(h_e, h_s, h_t, h_{st}, h_{ts}, h_{sts}, h_{tst}, h_{stst})$ as a $\mathbb{Z}[v^{\pm1}]$-module. The multiplication is such that $h_e$ commutes with $h_s$ and $h_t$, both $h_s$ and $h_t$ satisfy the quadratic relation of \Cref{Eq:Quadrel}, and there is the following braid relation:
		\begin{equation}
		h_sh_th_sh_t = h_th_sh_th_s \ .
		\end{equation}
	\end{example}

	\subsection{The Kazhdan--Lusztig basis}

	Let $s\in S$. One can check easily that the inverse of $h_s$ is $h_s+v-v^{-1}\in\He$.  Since the set of all $h_s$ for $s\in S$ generates $\He$ as a $\mathbb{Z}[v^{\pm1}]$-algebra, it follows that for every $w\in W$, the corresponding $h_w$ also admits an inverse. 
	
	The morphism of $\mathbb{Z}$-modules defined by
	\begin{equation}
	\iota: \left \{
	\begin{array}{rcl}
	\He & \rightarrow & \He \\
	v & \mapsto & v^{-1} \\
	h_w & \mapsto & (h_{w^{-1}})^{-1}
	\end{array}\right.
	\end{equation} 
	is a ring automorphism of $\He$. A cornerstone of Kazhdan--Lusztig theory is the following theorem (see \cite[Theorem 1.1]{kazhdan1979representations}):
	\begin{thm}[Kazhdan--Lusztig]
		For all $w\in W$ there exists a unique $\iota$-self-dual element of $\He$ of the form
		\begin{equation}\label{Eq:formKL}
		b_w = h_w + \sum_{z \leq w} h_{z,w} h_z 
		\end{equation}
		where $h_{z,w} \in v \mathbb{Z}[v]$. Moreover the set $\{b_w\}_{w\in W}$ is a basis of $\He$ as $\mathbb{Z}[v^{\pm 1}]$-module. 
	\end{thm}
	
	The \textit{Kazhdan--Lusztig polynomial} $p_{z,w}$ is defined as:
	\begin{equation}\label{Eq:KLpoly}
	p_{z,w} = v^{l(w)-l(z)} h_{z,w} \ .
	\end{equation}
	
	\begin{example}
		For $s\in S$, the element $b_s = h_s + v \in \He$ is always self-dual under $\iota$ and moreover it is of the form of \Cref{Eq:formKL}. It is clear that $b_e = h_e = 1$. Hence the Kazhdan--Lusztig basis of $\He_{(\mathfrak{S}_2,\{s\})}$ is the pair $(b_e = 1, b_s)$.
	\end{example}

	\paragraph{Positivity.}
	The Kazhdan--Lusztig basis enjoys numerous positivity properties which have been shown to be a consequence of a combinatorial Hodge theory in the category of Soergel bimodules \cite{williamson2016hodge}. For example the Kazhdan--Lusztig polynomials $p_{z,w}$ of \Cref{Eq:KLpoly} are positive (meaning that $p_{z,w}\in\mathbb{Z}_{\geq 0}[v]$), and the structure constants of the Hecke algebra are positive when expressed in the Kazhdan--Lusztig basis: if one sets
	\begin{equation}
	b_xb_y = \sum \tensor{\mu}{_x_y^z} b_z ~ ,
	\end{equation}
	then $\tensor{\mu}{_x_y^z}\in\mathbb{Z}_{\geq 0}[v^{\pm1}]$.

	\subsection{Ciliated surfaces}\label{Sec:ciliated-surface}
	
	A \textbf{ciliated surface} is an oriented topological surface obtained as follows: remove $n$ disjoint open disks labeled $1,\dots,n$ out of the oriented surface of genus $g$ with $k$ punctures, for $g,k,n\geq0$. On the $i$-th boundary circle add $p_i\geq1$ marked points called cilia, for $i=1,\dots,n$. 
	
	Topologically, a ciliated surface is determined by its genus, the integer $k$ and the set $\{p_1, ..., p_n\}$. We will denote such a surface $\Sigma_{g,k,\{p_1,...,p_n\}}$ in general, and simplify the writing to $\Sigma_{g,k}$ when it has no cilium, and to $\Sigma_{g,\{p_1,...,p_n\}}$ when it has no puncture. Let $c=\sum p_i$ be the total number of cilia.
	
	\begin{figure}[h!]
		\centering
		\includegraphics[scale=1.4]{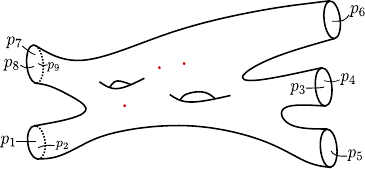}
		\caption{The ciliated surface $\S_{2,3,\{2,2,1,1,3\}}$.}
	\end{figure}
	
	In the sequel we will only consider ciliated surfaces such that $k + n \geq 1$, and such that when $g = 0$ either $k + n \geq 3$ or $k + n = 2$ and $c \geq 1$ or $k + n = 1$ and $c \geq 3$, for such ciliated surfaces can be triangulated. A \textit{triangulation} $\mathsf{T}$ of $\Sigma$ is a decomposition of $\Sigma$ into triangles such that every vertex of a triangle is either a cilium or a puncture. Edges of $\mathsf{T}$ belonging to the boundary of $\Sigma$ are said to be external and the others, internal. Let $\#F(\mathsf{T})$, $\#E(\mathsf{T})$, $\#E_0(\mathsf{T}) = c$ and $\#V(\mathsf{T}) = k + c$ be respectively the number of faces, edges (external and internal), internal edges and vertices of $\mathsf{T}$. The Euler characteristic of the closure $\overline\Sigma$ is
	\begin{equation}\label{Eq:eulerchar}
	\#F(\mathsf{T})-\#E(\mathsf{T})+\#V(\mathsf{T}) = 2-2g+n~,
	\end{equation}
	and since $\mathsf{T}$ is a triangulation:
	\begin{equation}\label{Eq:triangulation}
	3\#F(\mathsf{T}) = 2\#E(\mathsf{T})-\#E_0(\mathsf{T})~.
	\end{equation}
	From \Cref{Eq:eulerchar} and \Cref{Eq:triangulation} one deduces that:
	\begin{equation}
	\left.
	\begin{array}{l}
	\#E(\mathsf{T}) = 6g-6+2c+3(k+n)\\
	\#F(\mathsf{T}) = 4g-4+c+2(k+n)
	\end{array}~.
	\right.
	\end{equation}
	
	Except when $g=0$ and $(k,n)=(0,n)$ or $(k,n)=(1,n)$ the number of triangulations of a ciliated surface is infinite. However one can always reach any triangulation from a reference one in a finite number of \textit{flips} which consist of replacing the diagonal of a quadrilateral formed by two adjacent triangles with the other diagonal.
	
	In what follows we will speak of the \textit{boundary of a ciliated surface} to refer to the disjoint union of the boundary segments connecting two adjacent cilia.

	\section{Definition of the polynomial and first properties}\label{Sec:invpuncsur}
	
	We present our construction which associates Laurent polynomials to ciliated surfaces in a pedestrian way. A more abstract viewpoint comes in \Cref{Sec:TQFT}.
	
	\subsection{On the standard structure constants in Hecke algebras}\label{Sec:struccons}
	
	Let $(W,S)$ be a Coxeter system and $\He$ the corresponding Hecke algebra. In \Cref{Sec:Hecstand} we introduced the standard basis $\{h_w\}_{w\in W}$ of $\He$ as a free $\mathbb{Z}[v^{\pm 1}]$-module. Let $\He^*$ be the free $\mathbb{Z}[v^{\pm 1}]$-module dual to $\He$ with standard dual basis $\{h^w\}_{w\in W}$.
	
	By definition, the structure constants $\tensor{c}{_{xy}^z}$ are given by
	\begin{equation}
	\tensor{c}{_{xy}^z} = h^z(h_x \cdot h_y) \;\in \bb{Z}[v^{\pm 1}] 
	\end{equation}
	for $x,y,z \in W$. Let us set:
	\begin{equation}\label{Eq:metricintui}
	c_{xyz}:=\tensor{c}{_{xy}^{z^{-1}}} \ .
	\end{equation} 
	
	The notation of \Cref{Eq:metricintui} can be understood through the standard trace in the Hecke algebra. A \textbf{trace} on $\He$ is a $\mathbb{Z}[v^{\pm 1}]$-linear map $\tr:\He \rightarrow \mathbb{Z}[v^{\pm 1}]$ such that $\tr(hh^{'})=\tr(h^{'}h)$ for all $h,h^{'}\in\He$. A trace is said to be \textit{symmetrizing} if the map $h' \mapsto \tr(hh')$ is non-degenerate for all $h\neq 0$. The map
	\begin{equation}
	\tr\left(\textstyle\sum_{w\in W} c_w h_w\right) = h^e\left(\textstyle\sum_{w\in W} c_w h_w\right) = c_e
	\end{equation}
	is a symmetrizing trace on $\He$ called the \textbf{standard trace} \cite[Proposition 8.1.1]{geck2000characters}. It is easy to see that the standard trace on $\He$ satisfies: 
	\begin{equation}\label{Eq:trace-hx}
	\tr(h_xh_y) = \delta_{x, y^{-1}}~.
	\end{equation}

	\begin{prop}\label{Prop:cyclicinv}
		For all $x,y,z \in W$, one has
		\begin{equation}
		c_{xyz} = \tr (h_x h_y h_z) \ .
		\end{equation}
		In particular, this implies that $c_{xyz}$ is cyclically symmetric: $c_{xyz} = c_{yzx} = c_{zxy}$.
	\end{prop}
	
	\begin{proof}
		By definition of the structure constants one has:
		\begin{equation}
		h_xh_y = \sum_{z'} c_{xyz'} h_{z'^{-1}}~.
		\end{equation}
		Multiplying by $h_z$, taking the trace and using \Cref{Eq:trace-hx} yields:
		\begin{equation}
		\tr(h_xh_yh_z) = \sum_{z'}c_{xyz'}\delta_{z,z'} = c_{xyz}~,
		\end{equation}
		which concludes the proof.
	\end{proof}
	
	\subsection{An invariant for surfaces with punctures}\label{Sec:invariant}
	
	\paragraph{Triangles and structure constants.} 
	Consider a triangle with oriented edges, labeled with elements $x,y,z \in W$. Fix also an orientation of the triangle. We associate to the triangle the structure constant $c_{x^ay^bz^c} \in \mathbb{Z}[v^{\pm 1}]$ where $a$ is $1$ (respectively, $-1$) if the orientation of the edge labeled by $x$ is induced by the triangles orientation (respectively, if not), and \textit{mutatis mutandis} for the two other edges. \Cref{Fig:triangles} gives two examples.
	
	\begin{figure}[h!]
		\centering
		\includegraphics[height=4cm]{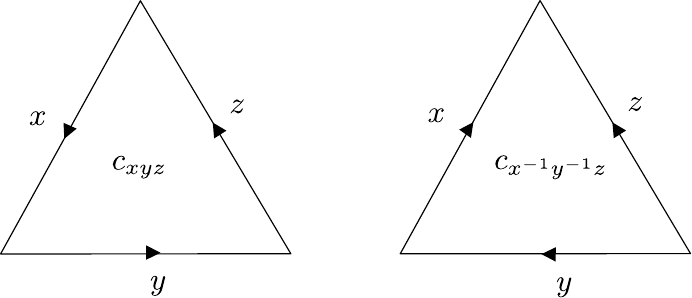}
		\caption{We associate $c_{xyz}$ to the left-most triangle and $c_{x^{-1}y^{-1}z}$ to the right-most one.}\label{Fig:triangles}
	\end{figure} 
	
	Note that by \Cref{Prop:cyclicinv} the quantity associated to the triangle is well-defined. In addition, it does not depend on the chosen orientation of the triangle:
	
	\begin{prop}\label{prop:revert-orientation}
		The structure constants satisfy $c_{xyz}=c_{z^{-1}y^{-1}x^{-1}}$.
		Thus, the associated quantity to a triangle is independent of its orientation.
	\end{prop}
	\begin{proof}
		Consider the following map $\sigma$ on $\He$, defined on the standard basis by 
		\begin{equation}
		\sigma(h_w) = h_{w^{-1}}
		\end{equation}
		and extended by linearity. We claim that $\sigma$ is an anti-involution of $\He$, i.e. that $\sigma(h_ah_b)=\sigma(h_b)\sigma(h_a) \;\forall a,b\in W$. Indeed it is enough to check this on the basic relations: the quadratic relation is invariant under $\sigma$, and for $w=sw'$ where $s\in S$ and $w'$ is a reduced word, we have 
		\begin{equation}
		\sigma(h_sh_{w'}) = \sigma(h_w) = h_{w^{-1}}=h_{w'^{-1}}h_s=\sigma(h_{w'})\sigma(h_s)~.
		\end{equation}
		By definition, we have 
		\begin{equation}
		h_xh_y=\sum_{z\in W} c_{xyz}h_{z^{-1}}~.
		\end{equation}
		Applying $\sigma$ gives 
		\begin{equation}
		h_{y^{-1}}h_{x^{-1}} = \sum_{z\in W} c_{xyz}h_z~.
		\end{equation}
		Hence $c_{xyz}=c_{y^{-1}x^{-1}z^{-1}}$. We conclude by cyclicity of $c_{xyz}$.
	\end{proof}
	
	\paragraph{Gluing triangles into ciliated surfaces.}
	
	Let $\mathsf{T}$ be a triangulation of the oriented topological surface $\Sigma_{g,k}$ of genus $g$ and with $k \geq 1$ punctures. 
	
	\begin{definition} Let $P_{g,k,W,\mathsf{T}}$ be the Laurent polynomial:
		\begin{equation}\label{Eq:defpolynomial}
		P_{g,k,W,\mathsf{T}}(v) = \sum \prod_{f\in\mathsf{T}} c_f(v) \in \mathbb{Z}[v^{\pm 1}] \ ,
		\end{equation}
		where the sum runs over all possible labelings of the edges of $\mathsf{T}$ by elements of $W$ and the product over all faces $f$ of $\mathsf{T}$, and where $c_f(v)$ is the structure constant associated to the face $f$ of $\mathsf{T}$ as before. 
		
		More generally, we can define such a polynomial for any ciliated surface with triangulation $\mathsf{T}$, as soon as each boundary component is labeled with an element of $W$. The sum in \Cref{Eq:defpolynomial} runs in this case over all possible labelings of internal edges of $\mathsf{T}$.
	\end{definition}
	
	In order to compute $P_{g,k,W,\mathsf{T}}$ one has to choose an orientation for the internal edges of $\mathsf{T}$ so that the $c_f(v)$ are well defined. Since we are summing over all possible labels of the edges, the polynomial $P_{g,k,W,\mathsf{T}}$ does not depend on these choices since changing the orientation amounts to replace $w\in W$ by its inverse. 
	
	The main point of our construction is the following:
	
	\begin{thm}\label{Thm:trianginv}
		The polynomial invariant does not dependent on the triangulation $\mathsf{T}$, hence it is a topological invariant of the ciliated surface. Further, it is preserved under reversing the orientation of the surface (and thus inverting all the boundary data).
	\end{thm}
	We denote the invariant by $P_{\S,W}$ for a ciliated surface $\S$ or by $P_{g,k,W}$ in case of a punctured surface $\S_{g,k}$. The theorem follows from the \emph{associativity} of the product in the Hecke algebra.
	
	\begin{proof}
	The second part is a direct consequence of \Cref{prop:revert-orientation} and the definition of $P_{g,k,W}$.
	
		For the first part, let $x,y,z \in W$. Then by definition:
		\begin{equation}
		h_xh_y = \sum_{w\in W} c_{xyw} h_{w^{-1}} \ ,
		\end{equation}
		hence
		\begin{equation}
		\sum_{w,v\in W} c_{xyw} c_{w^{-1}zv} h_{v^{-1}} = (h_xh_y)h_z = h_x(h_yh_z) = \sum_{t,v\in W} c_{yzt} c_{xt^{-1}v} h_{v^{-1}} \ ,
		\end{equation}
		which implies
		\begin{equation}\label{Eq:invunderflip}
		\sum_{w\in W} c_{xyw} c_{w^{-1}zv} = \sum_{t\in W} c_{yzt} c_{xt^{-1}v}
		\end{equation}
		for all $x,y,z,v \in W$. \Cref{Eq:invunderflip} can be described graphically as in \Cref{Fig:associativity}.
		
		\begin{figure}[h!]
			\centering
			\includegraphics[height=3cm]{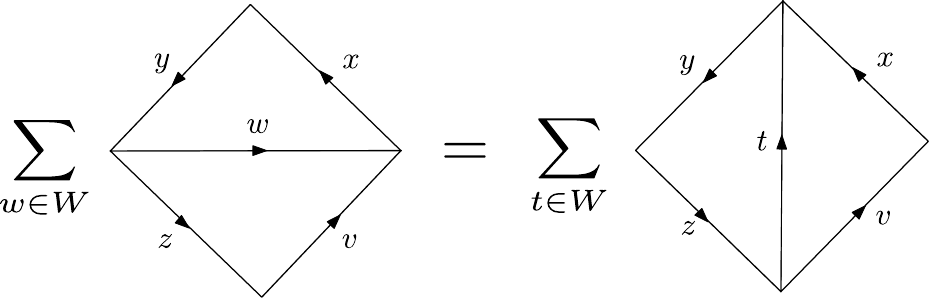}
			\caption{A consequence of associativity in Hecke algebras.}\label{Fig:associativity}
		\end{figure}
		
		Now since any two triangulations of $\Sigma_{g,k}$ can be related via a sequence of flips, \Cref{Eq:invunderflip} implies the proposition.
	\end{proof}

	\begin{example}\label{Ex:firstexamples}
		Let us give some examples of $P_{g,k,W}$. Later on we will come back to them and see how to compute them.
		\begin{itemize}
			\item $P_{0,3,\mathfrak{S}_2}(v) = P_{1,1,\mathfrak{S}_2}(v) = v^2+2+v^{-2}.$
			\item $P_{0,4,\mathfrak{S}_2}(v) = v^4+2v^2+2+2v^{-2}+v^{-4}.$
			\item $P_{0,3,\mathfrak{S}_3}(v) = v^6+2v^4+10v^2+10+10v^{-2}+2v^{-4}+v^{-6}.$
			\item $P_{1,1,\mathfrak{S}_3}(v) = v^6+2v^4+4v^2+4+4v^{-2}+2v^{-4}+v^{-6}.$
			\item $P_{\triangle_{x,y,z},W}(v) = c_{xyz}(v)$ where $\triangle_{x,y,z}$ stands for the triangle seen as the ciliated surface $\Sigma_{0,0,\{3\}}$ with labels $x, y, z \in W$ on the exterior edges assigned counterclockwisely.
		\end{itemize}
	\end{example}
	
	We give additional examples in \Cref{Appendix:compute} where we explain how to use Sage and the package CHEVIE of Gap3 to compute these polynomials.
	
	There are three remarkable observations to be done for punctured surfaces from these examples: the polynomials are functions of $v^{-2}=q$, they are invariant under $q\mapsto q^{-1}$ and have positive coefficients. The first two observations are actually properties that we will prove in \Cref{Prop:invariance-poly}, while we will analyze the positivity in \Cref{Sec:final-positivity}. 
	
	\begin{Remark}
		Jumping ahead a bit, since for punctured surfaces the invariant polynomial only depends on $v^{-2}=q$ we will use the variable $q$ in this case instead of $v$ to lighten the notation. For surfaces with cilia we keep the use of $v$.
	\end{Remark}

	\subsection{Gluing surfaces}
	
	The polynomial invariant behaves nicely under gluing of two surfaces $\Sigma_1$ and $\Sigma_2$ with same boundary data $D$ as in \Cref{Fig:quadri}. Schematically:
	\begin{equation}\label{gluing-1}
	P(\Sigma_1 \cup \Sigma_2) = \sum_D P(\Sigma_1,D)P(\Sigma_2,D)~,
	\end{equation}
	where $P(\Sigma,D)$ denotes the polynomial invariant for the surface $\Sigma$ with boundary data $D$. 
	
	\begin{figure}[h!]
		\centering
		\includegraphics[scale=0.7]{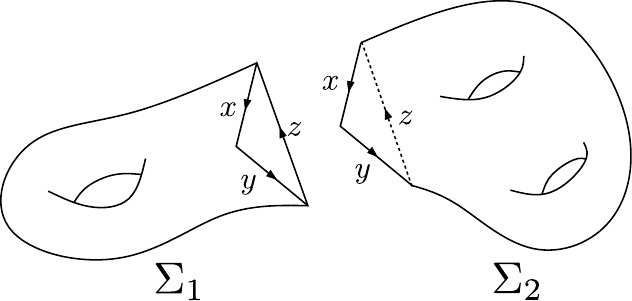}
		\caption{Gluing ciliated surfaces.}\label{Fig:quadri}
	\end{figure}
	
	We can use \Cref{gluing-1} to compute recursively the polynomial. For $\Sigma_1=\Sigma_{0,0,{3}}$ a triangle and $\Sigma_2=\Sigma_{g,k,\{3,...\}}$ a surface with a triangle boundary (that is, a circle component of its boundary with three cilia), \Cref{gluing-1} becomes:
	\begin{equation}\label{Eq:recursion1}
	P\left(\Sigma_{g,k,\{...\}},D\right) = \sum_{x,y,z \in W} c_{xyz} P\left(\Sigma_{g,k,\{3,...\}},D\cup\{x,y,z\}\right)~,
	\end{equation}
	where $D$ is the boundary data of the ciliated surface $\Sigma_{g,k,\{...\}}$ which may have non-empty boundary, and $D\cup\{x,y,z\}$ is the boundary data of $\Sigma_{g,k,\{3,...\}}$. Note that we used \Cref{prop:revert-orientation}.
	
	\begin{figure}[h!]
		\centering
		\includegraphics[scale=0.95]{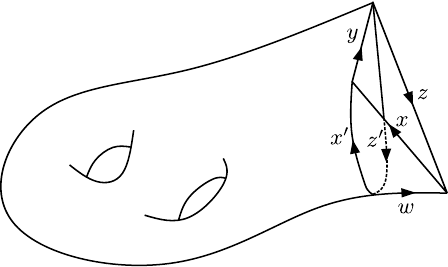}
		\caption{Reducing the number of punctures.}\label{Fig:reducepunct}
	\end{figure}
	
	For a ciliated surface with a triangle boundary carrying the labels $x,y$ and $z$, the gluing can be used to reduce the number of punctures. To do so, we use for $\Sigma_2$ a quadrilateral with boundary $z, x, x'^{-1}, z'^{-1}$. The surface after gluing is depicted in \Cref{Fig:reducepunct}. The polynomial invariant of this quadrilateral is given by $\sum_w c_{xx'^{-1}w}c_{zw^{-1}z'^{-1}}$. Hence we get from the gluing property:
	\begin{equation}\label{Eq:recursion2}
	P\left(\Sigma_{g,k+1}\backslash \triangle_{x,y,z}\right) = \sum_{x',z',w \in W} c_{xx'^{-1}w}c_{zw^{-1}z'^{-1}} P\left(\Sigma_{g,k} \backslash \triangle_{x',y,z'}\right)~,
	\end{equation}
	with the obvious generalizations to the general cases.

	Let us now study our invariants for $q=1$ since the Hecke algebra specializes to the group algebra $\C[W]$ in that case.

	\subsection{Invariants of punctured surfaces at $q=1$}
	
	For $q=1$ the Hecke algebra specializes to the group algebra $\C[W]$.
	
	\begin{prop}\label{Prop:q1}
		For a punctured surface $\Sigma_{g,k}$ the value at $q=1$ of the polynomial is:
		\begin{equation}\label{Eq:q=1}
		P_{g,k,W}(1) = (\# W)^{k-1} \times \#\{\text{solutions in } W \text{ to } \prod_{i=1}^g [a_i,b_i] = 1\} ~.
		\end{equation}
	\end{prop}
	
	In particular $P_{0,k,W}(1) = (\# W)^{k-1}$ which is in accordance with \Cref{Ex:firstexamples}.
	
	\begin{Remark}
		The right-most term in the right-hand-side of \Cref{Eq:q=1} can be expressed in terms of the characters of $W$ using Frobenius' formula (see \cite[theorem A.1.10 in the Appendix by Don Zagier]{lando2013graphs}). This yields: 
		\begin{equation}\label{poly-q1}
		P_{g,k,W}(1) = (\# W)^{2g-2+k} \times \sum_{\chi}\frac{1}{\chi(1)^{2g-2}}~.
		\end{equation}
		Notice the similarity to Burnside's formula for Hurwitz numbers (see \cite[Theorem 1.3]{gunningham2016spin} and the original paper \cite{hurwitz1891ueber}).
	\end{Remark}
	
	Our strategy to prove \Cref{Prop:q1} is to develop the surface $\Sigma_{g,k}$ as a $4g$-gon and to count explicitly the contributions.
	
	\begin{proof}
		We start with the case $k=1$ and $g>0$. The surface $\Sigma_{g,1}$ is obtained by gluing the sides of a $4g$-gon as described on the left of \Cref{Fig:q1-1} in the case $g=2$. We also show a triangulation $\mathsf{T}$. Note that the $4g$ vertices of the polygon correspond to a single point in $\Sigma_{g,1}$ which is the puncture.
		
		\begin{figure}[h]
			\centering
			\includegraphics[height=5cm]{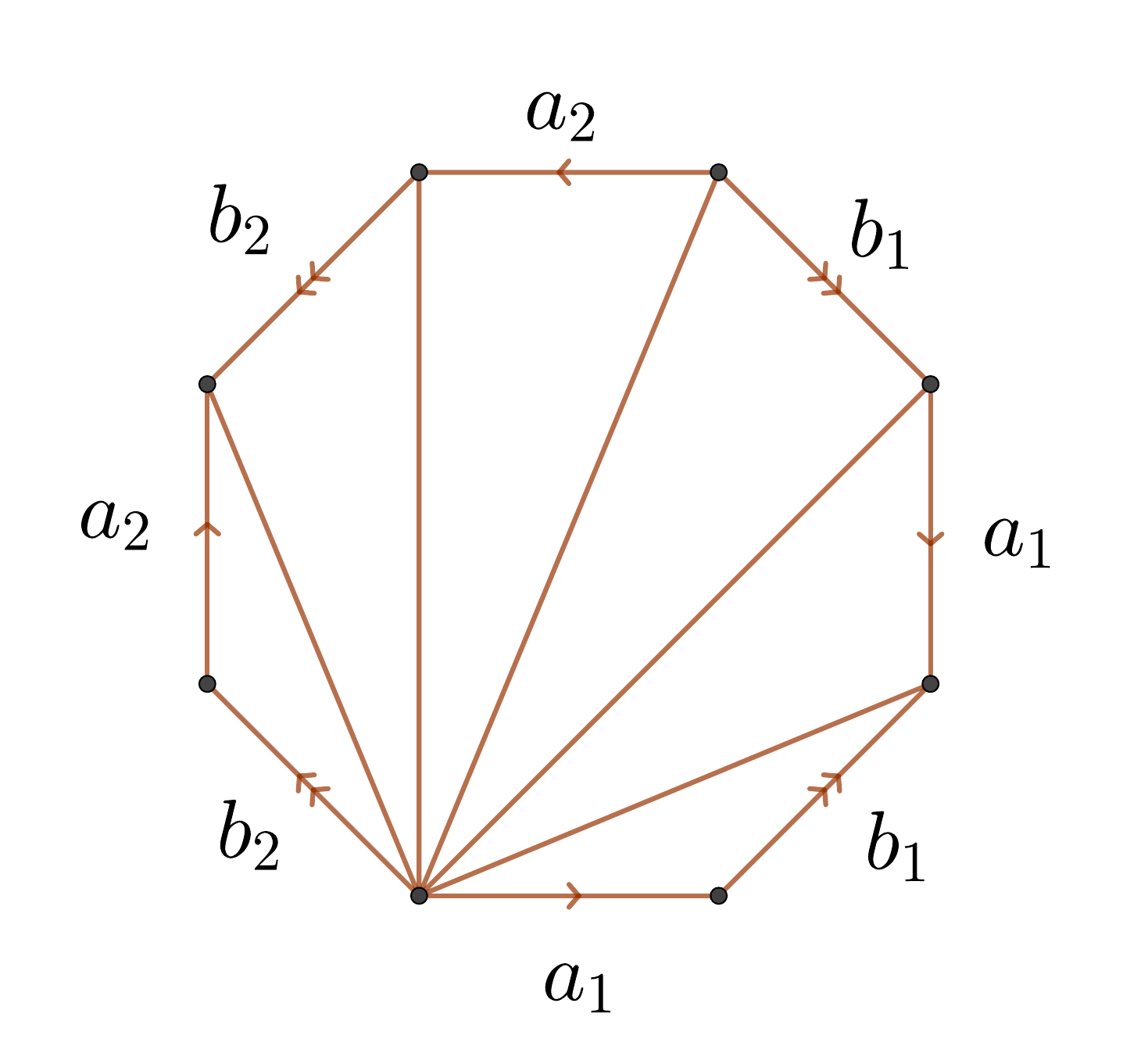}
			\includegraphics[height=5cm]{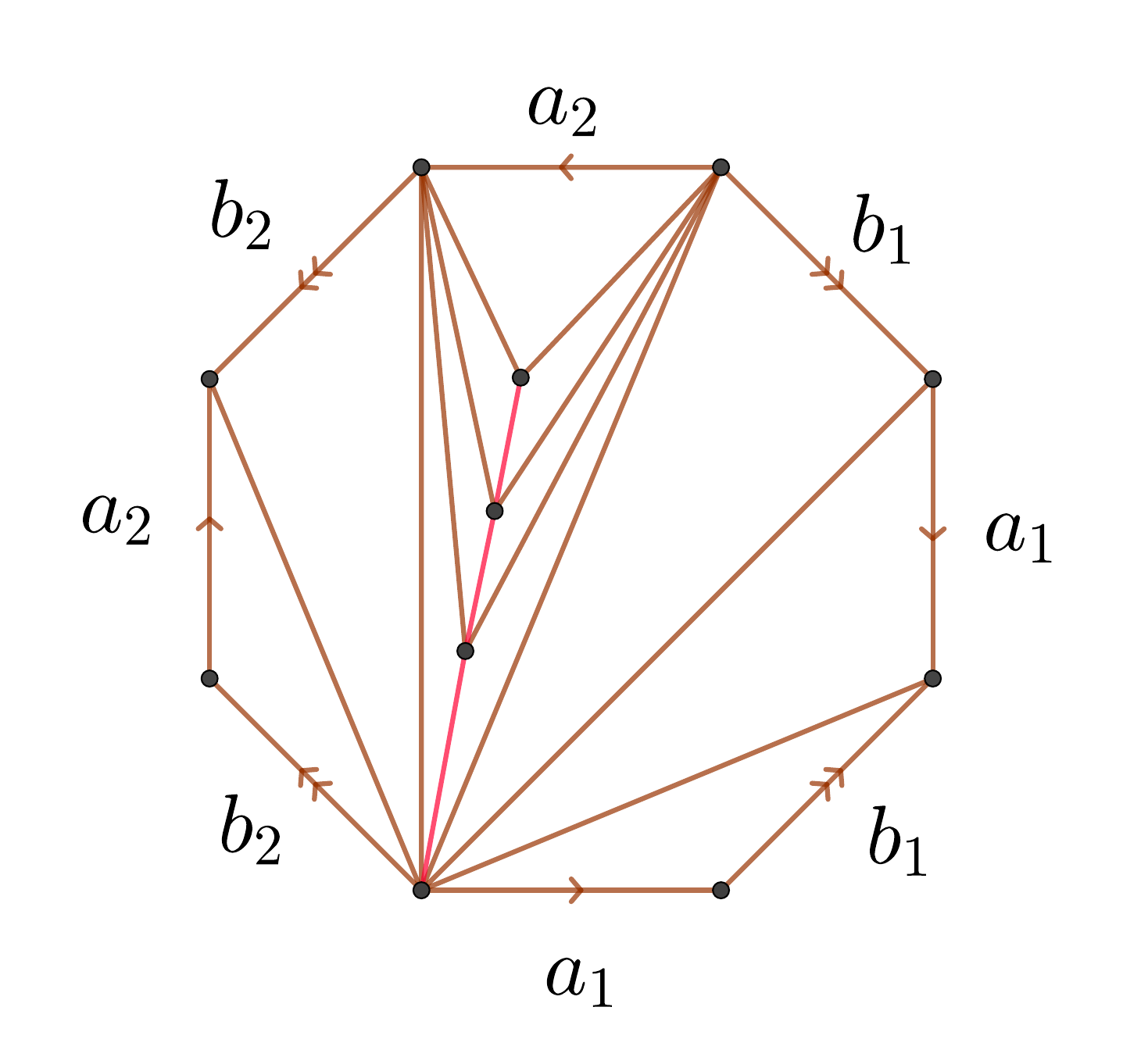}
			\caption{Gluing of $4g$-gon}\label{Fig:q1-1}
		\end{figure}
		
		The specializations of the structure constants $c_{xyz}$ of the Hecke algebra at $q=1$ are equal to 1 if $xyz = 1$ in $W$ and to zero otherwise. In order to get non-zero contributions to the polynomial invariant specialized at $q=1$, the label on a side of a face $f$ of $\mathsf{T}$ is determined by the labels on the two other sides of $f$. Let now $a_1, b_1, ..., a_g, b_g \in W$ be the labels assigned to the edges of the $4g$-gon. Let $\mathsf{T}$ be a triangulation of the $4g$-gon whose inner edges are all incident to a fixed vertex as on the left of \Cref{Fig:q1-1}. The labels on the inner edges of $\mathsf{T}$ are completely determined by the $a_i$ and $b_i$, otherwise the corresponding contribution to the invariant vanishes. 
		
		There is the following additional constraint:
		\begin{equation}\label{Eq:standard}
		\prod_{i=1}^g [a_i,b_i] = 1 ~ ,
		\end{equation}
		which can be understood as the consequence of the fact that the products of labels assigned to each triangle whose boundary is oriented counterclockwise has to be $1$ if one considers all the triangles glued together along edges as on the left of \Cref{Fig:q1-1}.
		
		Every solution of \Cref{Eq:standard} has contribution $1$ to the polynomial, and all other choices of $a_i$ and $b_i$ do not contribute. Therefore we get the proposition in the case $k=1$.
		
		For $k>1$, we add in our picture $k-1$ marked points in the middle of the polygon. We complete to a triangulation as shown on the right of \Cref{Fig:q1-1}. In particular there is a path of length $k-1$ connecting all $k$ punctures shown in red in \Cref{Fig:q1-1}. To each edge of this path, we associate a new label $x_i \in W$. One easily checks that these $x_i$ together with the $a_j$ and $b_j$ from above uniquely determines the labels assigned to each edge of the triangulation. The only relation is \Cref{Eq:standard} from above. Having fixed the data on the boundary of the $4g$-gon we have $(\# W)^{k-1}$ free choices which contribute by $1$ to our polynomial.
		
		Eventually for $g=0$ we represent $\Sigma_{0,k}$ as the gluing of a $2k-2$-gon with boundary $a_1, a_1^{-1}, a_2, a_2^{-1}, ..., a_{k-1}, a_{k-1}^{-1}$ with a triangulation similar as the ones of \Cref{Fig:q1-1}. The labels on the inner edges are determined by the boundary data and the constraint on the boundary is always satisfied. So we get $P_{0,k,W}(1) = (\# W)^{k-1}$.
	\end{proof}

	\section{Graphical calculus}\label{Sec:graph-calcul}
	
	In this section, we present a diagrammatic interpretation of our polynomial. In particular we give a graphical way to multiply elements in the Hecke algebra. In this section, we work with a Coxeter system $(W,S)$ associated to a Weyl group.
	
	We start from the observation that the structure constants in the Hecke algebra are linked to configurations of flags. From that, we define graphs with labeled edges which count the contributions to our polynomial.

	\subsection{Structure constants and flag counting}\label{Sec:counting}
	
	There is a well-known link between the structure constants in the Hecke algebra and triples of flags (see e.g. \cite[Proposition 2.2]{curtis1988representations}). We present a short way to get this link. 
	
	In \Cref{Sec:Hecstand}, we introduced the Hecke algebra using generators and relations. However the original definition by Iwahori \cite{iwahori1964structure} is geometric. Consider a finite field $\mathbb{F}_q$, a simple Lie group $G$ over $\bb{F}_q$ and fix a Borel subgroup $B$ of $G$. The Hecke algebra $\He^q_G$ is the algebra of functions on $G$ invariant under left and right shift by $B$. Multiplication is given by the convolution divided by the order of the Borel subgroup $\# B$. 
	Note that the double quotient $B\backslash G/B$ is in bijection with $W$. This is why the Hecke algebra has a basis $(T_w)_{w\in W}$ parameterized by the Weyl group. In this geometric approach the generators $(T_s)_{s\in S}$ satisfy the braid relations and the quadratic relation reads:
	\begin{equation}
	T_s^2 = (q-1)T_s+q~.
	\end{equation}
	
	The structure constants $\tensor{C}{_{xy}^{z}}$ are defined by 
	\begin{equation}
	T_xT_y = \sum_{z\in W} \tensor{C}{_{xy}^{z}}(q) T_z~.
	\end{equation}
	
	From the definition of the convolution product, it follows that 
	\begin{equation}
	\tensor{C}{_{xy}^{z}}(q) = \# \{ h\in x\mid h^{-1}g\in y\} / \# B~,
	\end{equation}
	where we interpret $x, y$ and $z$ as elements of $W\cong B\backslash G / B$ and where $g$ is any representative of $z\in B\backslash G/B$.
	
	The quotient $G/B$ is called the \textit{flag variety} of $G$. For $G=\GL_n(\C)$ this is the space of complete flags in $\C^n$.
	Recall the isomorphism $G\backslash (G/B)^2 \cong B\backslash G /B$ given by $h\mapsto (F_0,hF_0)$ where $F_0 \in G/B$ is the class of the unit element in $G$. Combining this isomorphism with $B\backslash G /B \cong W$, we see that the relative positions of two flags is described by the Weyl group.
	Using this isomorphism we can rewrite the formula for the structure constants as
	\begin{equation}
	\begin{split}
	\tensor{C}{_{xy}^{z}}(q) &= \# \{ h\in G\mid (F_0,hF_0) = x, (F_0,h^{-1}gF_0) = y\} / \# B \\
	&= \# \{ h\in G\mid (F_0,hF_0) = x, (hF_0,gF_0) = y\} / \# B ~.
	\end{split}
	\end{equation}
	
	Taking into account that $B$ acts freely on the right on the set $\{h\in x \mid h^{-1}g\in y\}$ and denoting $gF_0$ by $F_2$ and $hF_0$ by $F_1$, we finally get:
	
	\begin{prop}\label{Prop:structure-const-flags}
		The structure constants $\tensor{C}{_{xy}^{z}}(q)$ count the configuration of three flags with prescribed relative positions:	
		\begin{equation}
		\tensor{C}{_{xy}^{z}}(q) = \# \{ F_1 \in G/B \mid (F_0,F_1) = x, (F_1,F_2) = y\}~,
		\end{equation}
		where $(F_0,F_2) = z$.
	\end{prop}
	
	\begin{example}
		Let $G=\SL(2)$. Then $G/B = \mathbb{P}^1, W=\{1,s\mid s^2=1\}$, $\tensor{C}{_{11}^{1}}=1, \tensor{C}{_{11}^{s}}=\tensor{C}{_{s1}^{1}}=\tensor{C}{_{1s}^{1}}=0$ and $\tensor{C}{_{1s}^{s}}=\tensor{C}{_{s1}^{s}}=1$ obviously. To compute $\tensor{C}{_{ss}^{1}}(q)$, choose $F_0=F_2=\infty$. Then $\tensor{C}{_{ss}^{1}}(q)=\# \{p\in \bb{P}^1(\mathbb{F}_q)\mid p\neq \infty \} = q$. To compute $\tensor{C}{_{ss}^{s}}(q)$, choose $F_0=0$ and $F_2=\infty$. Then $\tensor{C}{_{ss}^{s}}(q) = \# \{p\in \bb{P}^1(\mathbb{F}_q) \mid p\neq 0,\infty\} = q-1$.
	\end{example}

	\subsection{Finite higher laminations for triangles}
	
	In this subsection, we present a graphical way to compute the product in the Hecke algebra in the standard basis. Our graphs are very similar to the ones that appear in the article \cite{elias2016soergel} by Elias--Williamson, in the context of the categorification of the Hecke algebras in terms of Soergel bimodules.

	\subsubsection{Definition}\label{Sec:defhighlam}
	
	We use the normalized version of the Hecke algebra where the quadratic relation reads $h_s^2 = 1+Qh_s$ where $Q=v-v^{-1}=q^{1/2}-q^{-1/2}$.
	
	The idea is to use the interpretation of the structure constants as triples of flags from the previous subsection. The transition between two flags with given relative position is decomposed into elementary moves. This corresponds to a decomposition of a product into elementary ones involving simple elements only. For example the quadratic relation is graphically given by the following picture:
	\begin{figure}[h!]
		\centering
		\includegraphics[height=3.3cm, trim= 0 100 0 100, clip]{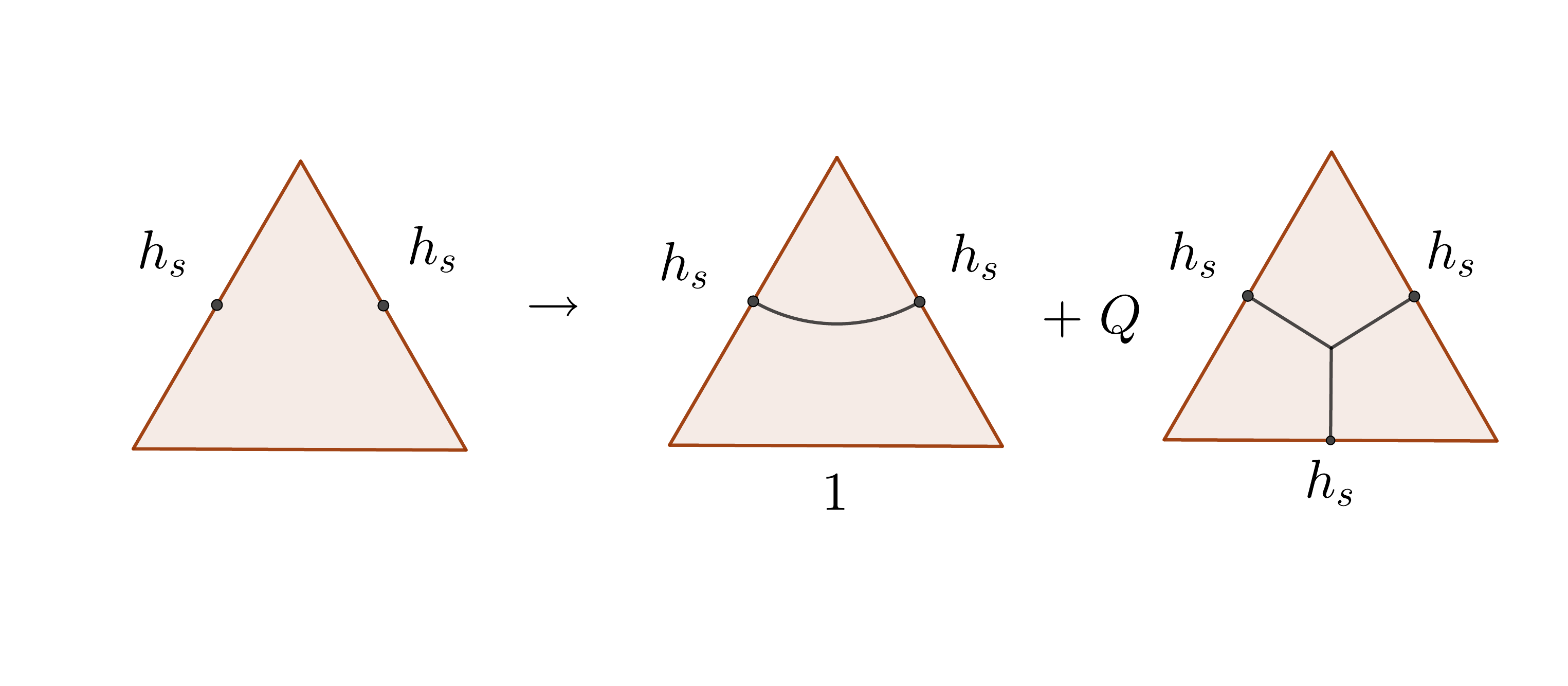}
		\caption{Graphical multiplication in Hecke algebra}\label{Fig:basic-example}
	\end{figure}
	
	More generally for $x,y \in W$ the product $h_xh_y$ is represented as a formal sum of graphs with edges labeled by simple reflections.
	First one has to fix reduced expressions of $x$ and $y$, using the simple reflections. Each simple element $h_s$ is represented graphically by an edge labeled by $s\in S$. 
	Elementary computations in $\He$ are either braid or quadratic relations, hence the vertices of the graph are of two types:
	\begin{itemize}
		\item[$\bullet$] \textit{trivalent} with the three edges carrying the same label, called \textbf{ramification point},
		\item[$\bullet$] \textit{of braid type} at the crossing of edges carrying the labels $s$ and $t$ such that $(st)^{m}= e$ in $W$. In that case there are $m$ incident edges of type $s$ and $m$ of type $t$ which alternate (the cases $m=2$ and $m=3$ are drawn in the middle and on the right of \Cref{Fig:vertex-types}).
	\end{itemize}
	
	\begin{figure}[h!]
		\centering
		
		\includegraphics[height=4cm, trim= 50 50 50 100, clip]{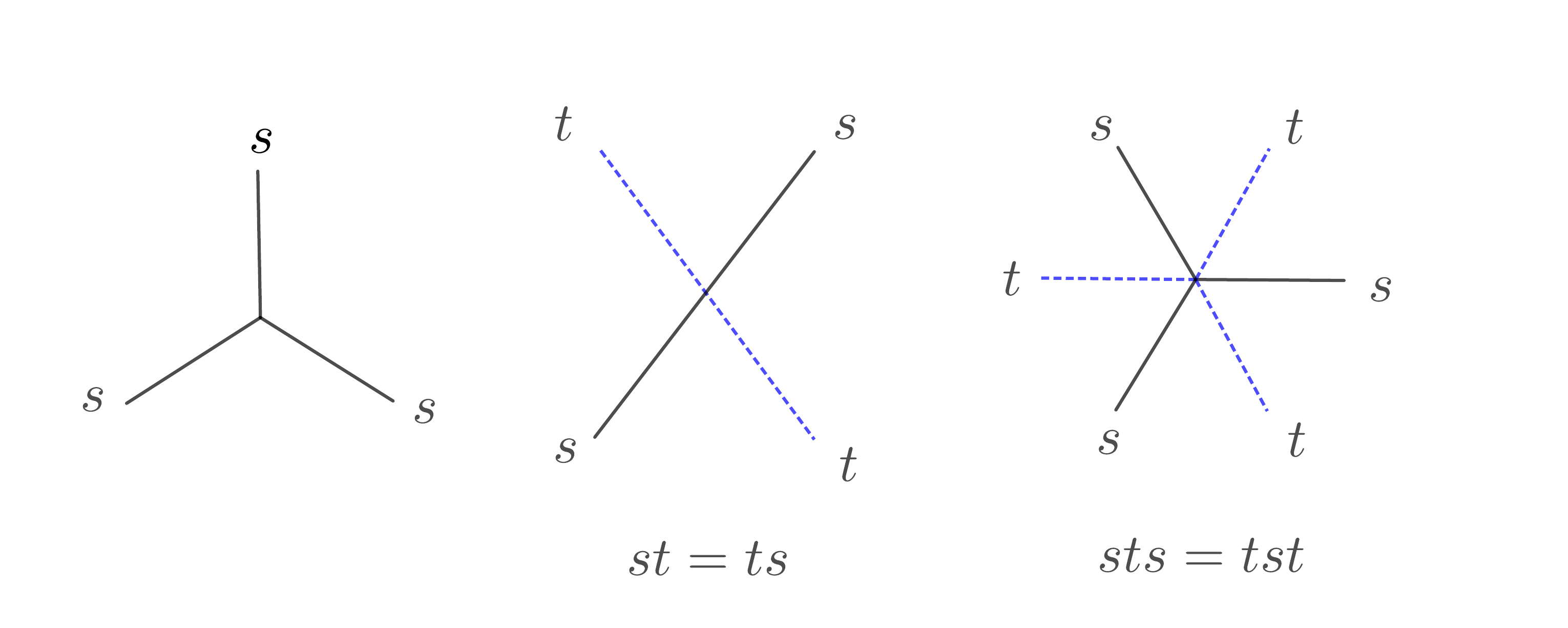}
		
		\caption{Vertex types} \label{Fig:vertex-types}
	\end{figure}

	\begin{definition}\label{Def:lamination}
		A \textbf{finite higher lamination of type $(W,S)$} (in short higher lamination) on a triangle $t$ is an equivalence class of planar graphs $\Gamma\subset t$ with edges labeled by elements in $S$, and satisfying the following criteria:
		\begin{enumerate}
			\item each vertex is either trivalent or of braid type,
			\item the edges of $\Gamma$ intersect the boundary of $t$ transversally,
			\item reading the labels along an edge of $t$ gives a reduced word in $W$,
			\item the graph is minimal in the sense described below.
		\end{enumerate}
		The equivalence class is generated by isotopy and the relations (1.1), (2.1), (2.2), (3.1) and (3.2) described below in \Cref{Sec:graphic-relations}.
	\end{definition}
	
	\begin{Remark}\label{Rem:higherlaminations}
		We call these graphs finite higher laminations because they generalize rational bounded measured laminations as described in \cite{fock2007dual}. The rough idea is that one recovers rational bounded measured laminations from the higher laminations for the affine Weyl group $\widehat{A}_1 = \langle s,t \mid s^2=t^2=1 \rangle$ after removing the singular leaves. Details of this correspondence and higher laminations corresponding to affine Coxeter systems will be discussed in details in a forthcoming publication. Aspects of higher laminations as tropical points of higher Teichmüller spaces and generalizations of rational measured laminations have been studied in \cite{xie2390higher} and \cite{le2016higher}.
	\end{Remark}
	
	\paragraph{Minimality.}
	Let $\Gamma$ be an $S$-labeled graph on an oriented triangle $t$ satisfying the first three conditions of \Cref{Def:lamination}, and let $x, y, z$ be the elements of $W$ corresponding to the reduced words on the sides of $t$.
	A configuration of flags on $t\backslash \Gamma$, one flag for each face, is called \textbf{valid} if the relative position of two adjacent flags is given by the simple element which appears as the label on the edge between the two corresponding faces.
	
	\begin{definition}
		The graph $\Gamma$ is \textbf{minimal} if for all triple of flags, one flag for each vertex of $t$, there is at most one valid configuration of flags on $t\backslash \Gamma$ which extends the triple of flags.
	\end{definition}
	
	An example of a non-minimal graph is given in \Cref{Fig:non-minimal}: the flag corresponding to the inner component of $t\backslash \Gamma$ is not uniquely determined by the flags around.
	
	\begin{figure}[h!]
		\centering
		\includegraphics[height=3.5cm]{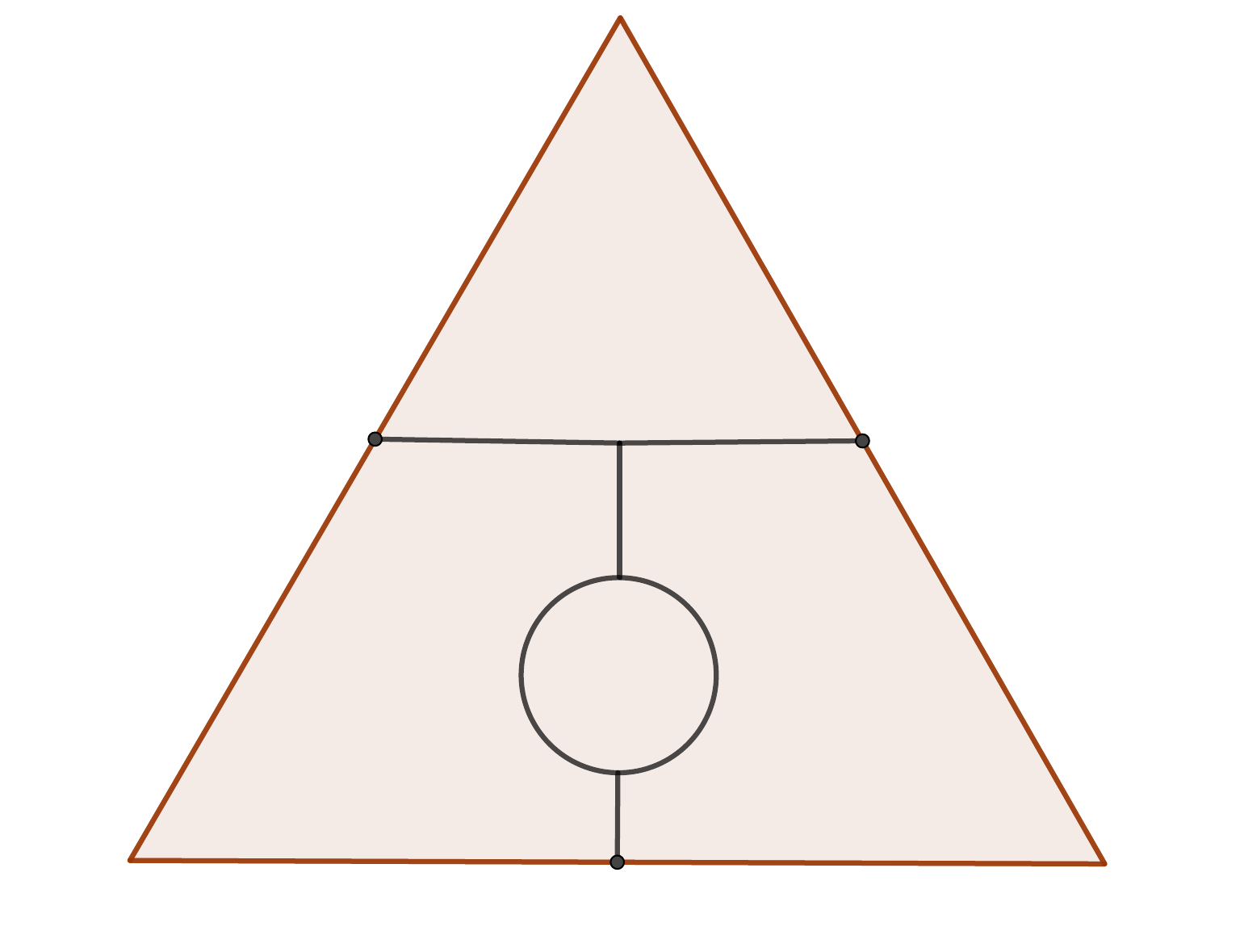}
		\caption{Example of a non-minimal graph}\label{Fig:non-minimal}
	\end{figure}
	
	\begin{Remark}
		A simple criterion to decide whether a graph is minimal or not is still lacking. In particular, we would like to find a minimality criterion for general Coxeter systems $(W,S)$ for which flags are not defined.
	\end{Remark}

	\subsubsection{Relations}\label{Sec:graphic-relations}
	
	Given a triple of flags, there might be several minimal labeled graphs realizing the configuration. In other words, the decomposition into elementary moves between flags is not unique. For example, given two flags in relative position $w\in W$ any reduced expression for $w$ is a working decomposition into elementary moves. Hence we have to quotient out by relations in order to count each triple of flags only once in \Cref{Prop:structure-const-flags}. We describe here a set of relations that we conjecture to be complete.
	
	There are relations involving one, two and three different labels (also called ``colors''). The one-color relation is given by the following picture:
	\begin{figure}[h!]
		\centering
		\includegraphics{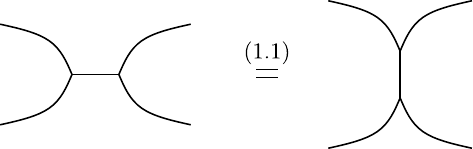}
		\caption{One-color relation}\label{Fig:one-color-relation}
	\end{figure}
	
	The two-color relations are of two kinds: the first states that one can simplify two neighbor vertices when they are of the same braid type, as shown on the left of \Cref{Fig:two-color-relation} for $m=3$ (for $m=2$ this gives the second Reidemeister move). The second relation describes how to glide a trivalent vertex through a vertex of braid type and is shown on the center (resp. right) of \Cref{Fig:two-color-relation} for $m=2$ (resp. $m=3$). 
	\begin{figure}[h!]
		\centering
		\includegraphics{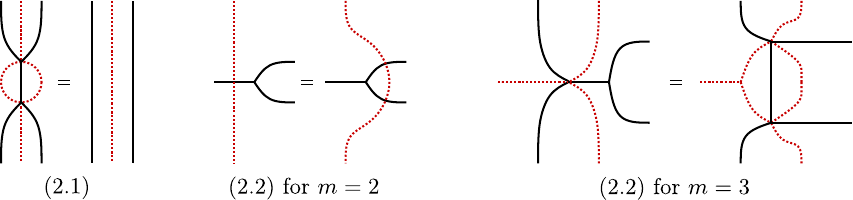}
		\caption{Two-color relation}\label{Fig:two-color-relation}
	\end{figure}
	
	Eventually there is a three-color relation for each parabolic subgroup of rank 3 in $W$, displayed in \cite[Section 5]{elias2016soergel}. Two examples are presented in \Cref{Fig:three-color-relation}: the upper-one is the relation corresponding to a subgroup of type $A_1\times A_1\times A_1$ (it is the third Reidemeister move) and the lower-one corresponds to a subgroup of type $A_3$.
	\begin{figure}[h!]
		\centering
		\includegraphics{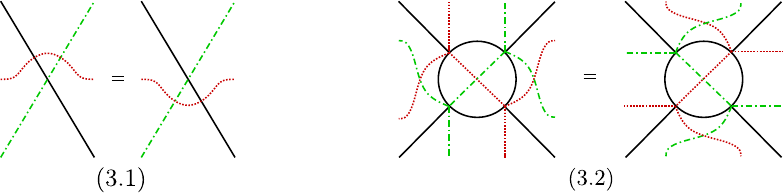}
		\caption{Three-color relation}\label{Fig:three-color-relation}
	\end{figure}
	
	\begin{Remark}
		All the relations of \cite{elias2016soergel} which do not imply a loose end are relations for our higher laminations. It seems there is a link between the graphical calculus in the Hecke algebra we introduced and the graphical calculus of Elias--Williamson describing morphisms between Bott--Samelson bimodules, even we do not understand the correspondence yet.
	\end{Remark}
	
	\Cref{Fig:example-relations} shows the equivalence of two seemingly different graphs through these relations.
	
	\begin{figure}[h!]
		\centering
		\includegraphics{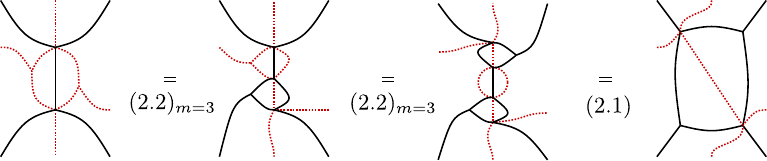}
		\caption{An example of applying relations}\label{Fig:example-relations}
	\end{figure}
	
	The two-color relation (2.1) and all three-color relations give a complete list of relations for reduced expressions in $W$: two reduced expressions of an element of $W$ can be related through a finite number of these relations \cite[Chapter 2, §5]{ronan2009lectures}. This is enough to show that the latter are all the relations we need in the case of higher laminations without ramification points.
	
	In general, we conjecture the following:
	\begin{conj}\label{Conj:complete-relations}
		The relations of above are complete: two minimal labeled graphs corresponding to the same triple of flags on the vertices of the triangle can be related through a finite sequence of them.
	\end{conj}
	
	The conjecture seems reasonable since we only have to look for relations involving ramification points. Our one-color relation describes how two ramification points interact and the two-color relation (2.2) describes the interaction between a ramification point and a vertex of braid type. It does not seem too presumptuous to expect that these are the only cases one needs to consider.

	\subsubsection{Existence}
	
	We show that we can associate a set of representatives $\Gamma$ of higher laminations to any multiplication in the Hecke algebra. The non-trivial part is to show minimality, that is, the existence of a unique configuration of flags on the connected components of $t\backslash \Gamma$ (henceforth called the \textit{faces} of $\Gamma$).
	
	Let $\Gamma$ be a representative of a higher lamination. Recall that an assignment of flags to the faces of $\Gamma$ is \textit{valid} if the relative position of any two adjacent flags is the label of the edge between the corresponding faces.
	
	We start with some easy results, the proof of which is left as an exercise for the interested reader.
	\begin{lemma}
		Let $v,w \in W$ such that $vw$ is not reduced. Then there are reduced expressions for $v$ and $w$ of the form $v=v's$ and $w=sw'$ where $s\in S$ is a simple reflection.
	\end{lemma}
	
	\begin{coro}\label{Coro:case-three}
		For $s,t \in S$, consider a reduced expression for $w\in W$ of the form $w=w'u$ where $u$ is a word in the letters $s$ and $t$ with maximal length. Denote by $\bar{u}$ the word obtained from $u$ by exchanging $s$ and $t$. Then $w'\bar{u}$ is also reduced.
	\end{coro}

	The two other lemmas concern the extension of flag configurations.
	
	\begin{lemma}\label{Lem:flags-opp}
		Given two flags in opposite faces with respect to a vertex a braid type, and whose relative position is compatible with the local structure around the vertex, there is a unique way to associate a valid configuration of flags around the vertex extending the initial data.
	\end{lemma}
	In type $A$ this can be verified easily by a computation in $\mathbb{P}^2$. 
	
	\begin{lemma}\label{Lem:flags-line}
		Given two flags $F_0$ and $F_1$ in relative position $w$ and a reduced word $w = s_{i_1}\cdots s_{i_k}$, there is a unique sequence of flags $F_0 = F_{i_1}, F_{i_2}, ..., F_{i_{k+1}} = F_1$ such that $(F_{i_l},F_{i_{l+1}}) = s_{i_l}$ for all $l=1,...,k$.
	\end{lemma}
	
	Any product $h_xh_y$ in $\He$ of elements of the standard basis can be decomposed into elementary moves which are either a quadratic relation or a braid relation, and we have diagrams for these. By juxtaposition we get a collection of graphs corresponding to the product $h_xh_y$. 
	
	\begin{prop}\label{Prop:higher-lam-existence}
		The graphs obtained from expressing $h_xh_y$ in the standard basis of $\He$ are representatives of higher laminations.
	\end{prop}
	
	\begin{figure}[h!]
		\centering
		
		\includegraphics[height=4cm]{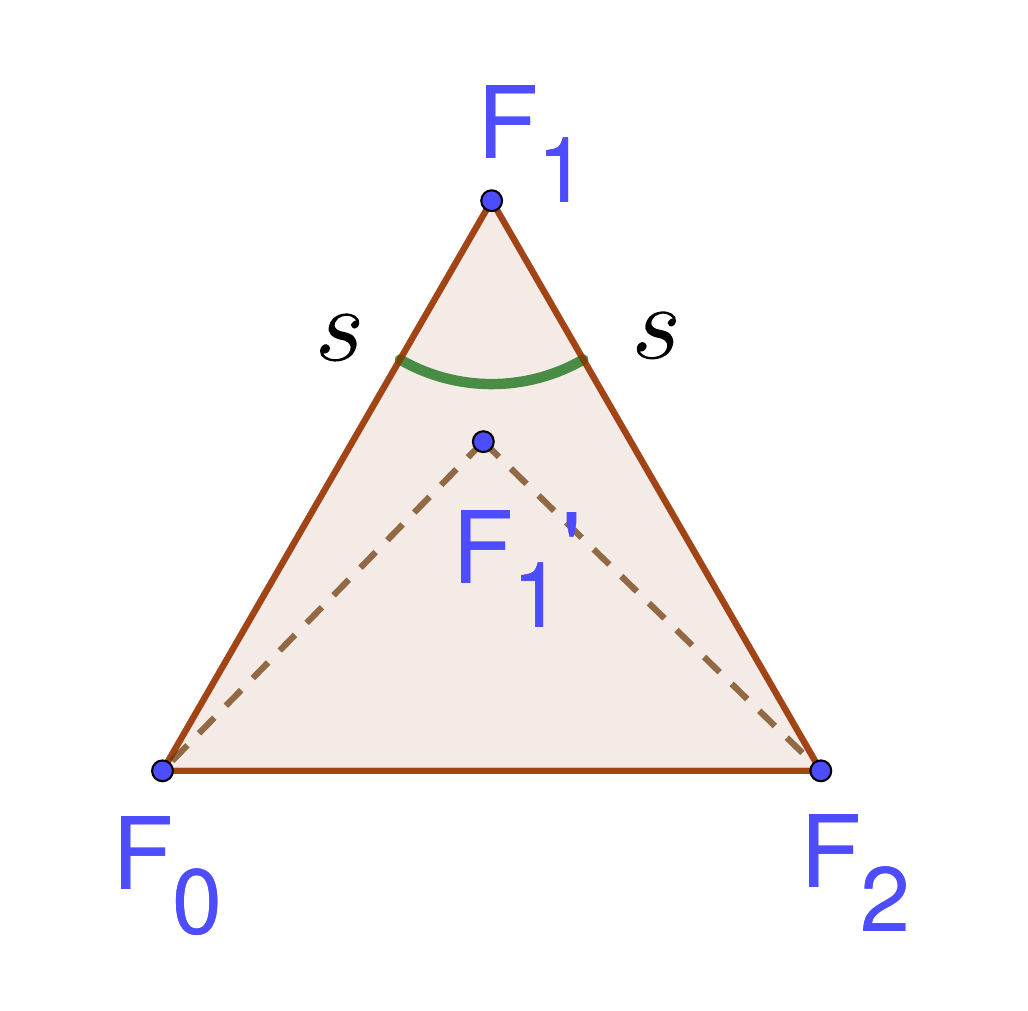} \hspace{0.5cm}
		\includegraphics[height=4cm]{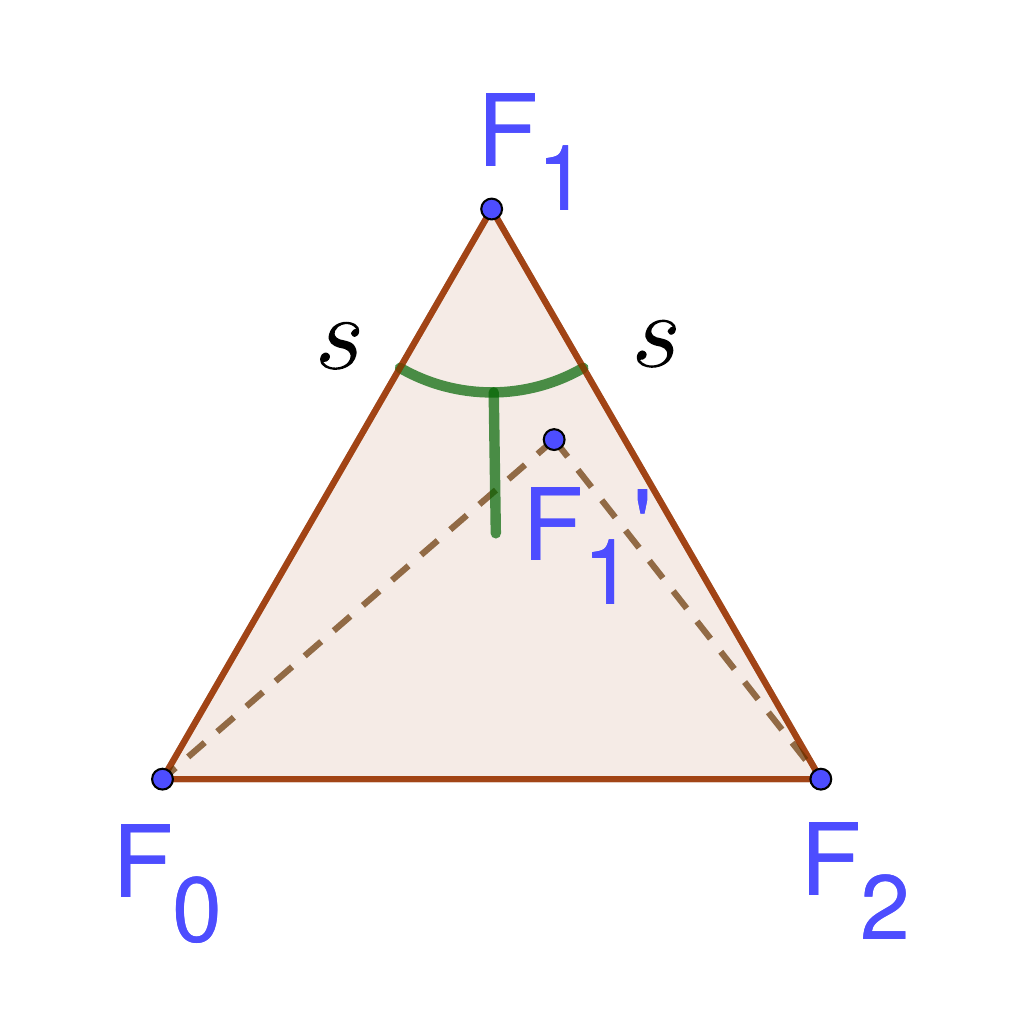} \hspace{0.5cm}
		\includegraphics[height=4cm]{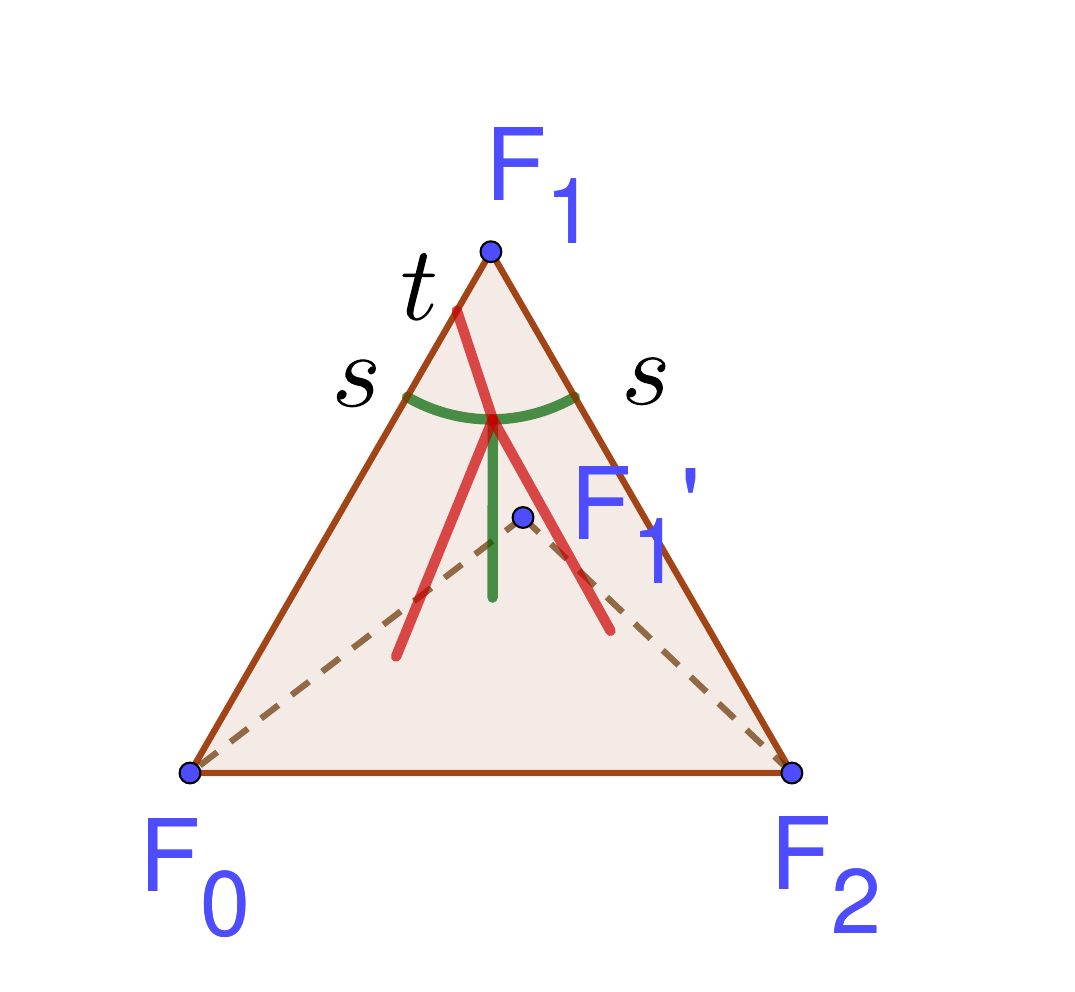}
		
		\caption{Local structure around $F_1$. Left: simple edge. Middle: ramification point. Right: braid vertex.}\label{Fig:three-cases}
	\end{figure}

	\begin{proof}
		Let $x,y\in W$ and $h_x, h_y$ be the corresponding elements in $\He$. The latter can be written as a product of the $h_{s_i}$ corresponding to the elements in $S$. Let us fix such a decomposition for $h_x$ and $h_y$. As one computes the product $h_xh_y$ in $\He$, one uses braid relations and quadratic relations until all the terms one obtains are products of $h_{s_i}$ corresponding to reduced words in $W$. Each quadratic relation increments the number of terms in the expression by one. Let us consider one of the terminal terms and assume it corresponds to $h_z$ for $z \in W$. Let $\Gamma$ be the corresponding graph on a triangle $t$ whose sides are respectively labeled by the reduced words corresponding to $y$, $x$ and $z^{-1}$ as one reads the boundary counterclockwise.

		By \Cref{Prop:structure-const-flags}, we can associate a triple of flags $(F_0, F_1, F_2)$ to the vertices of $t$ such that $(F_2,F_1)=y$, $(F_1,F_0)=x$ and $(F_2,F_0)=z$ in $G\backslash (G/B)^2$.
		
		The only non-trivial fact to check is that $\Gamma$ is minimal, i.e. that we can extend the triple of flags to a unique valid configuration. We prove this by induction on the number of faces of $\Gamma$. The initialisation with one face is trivial since it corresponds to $1\times 1 = 1$. 
		
		In general, by \Cref{Lem:flags-line} we have flags assigned to all boundary faces. We can assume that the local configuration around the vertex with flag $F_1$ is given by one of the three cases shown in \Cref{Fig:three-cases}. This is because away from that vertex, we can only apply braid relations (since $x$ and $y$ are reduced) which amount to choose another reduced expression for $x$ or $y$.
		
		
		\medskip
		\underline{Case 1:} There is a simple edge next to $F_1$.
		\smallskip
		
		\noindent The region just this edge $F_1$ already has an associated flag $F_1'$ since it is a boundary face. The restriction of $\Gamma$ to the triangle $(F_0,F_1',F_2)$ has strictly less faces and we can choose the boundary of this new triangle so that it is still a higher lamination: first one can clearly assume that $\Gamma$ intersects the boundary of the new triangle transversely and the sides chosen close to the sides of the original triangle so that the words assigned to them are the words on the sides of the original triangle with the letter corresponding to the edge between $F_1$ and $F_1'$ removed. Thus we can apply the induction hypothesis.
		
		\medskip
		\underline{Case 2:} There is a ramification point next to $F_1$.
		\smallskip
		
		\noindent As in case 1, the two flags below $F_1$ are determined by \Cref{Lem:flags-line}. Let $F_1'$ be one of these flags. The restriction of $\Gamma$ to $(F_0,F_1',F_2)$ has strictly less faces, we can again assume that it is a higher lamination and hence apply the induction hypothesis.
		
		\medskip
		\underline{Case 3:} There is a braid vertex next to $F_1$.
		\smallskip
		
		\noindent Since all boundary faces are already assigned a flag, there are two opposite faces around the braid vertex carrying a a flag and this assignment is consistent with the local structure of $\Gamma$. By \Cref{Lem:flags-opp}, we can uniquely extend this configuration around all the regions around the braid vertex. Let $F_1'$ be the flag below the braid vertex as shown on the right of \Cref{Fig:three-cases}. 
		Let us show that we can apply the induction hypothesis to the restriction of $\Gamma$ to the triangle formed by $(F_0,F_1',F_2)$. The only non-trivial fact to check is that the expressions induced on the edges $(F_0,F_1')$ and $(F_1',F_2)$ are reduced. By changing the reduced expression for $x$ or $y$, we can assume that the number of incoming edges in the braid vertex is maximal.
		The fact that the restriction of $\Gamma$ is reduced then follows from \Cref{Coro:case-three}.
	\end{proof}
	
	Given a representative $\Gamma$ of a higher lamination one wonders how many triples of flags modulo $G$ correspond to it.
	Let $\ram(\Gamma)$ be the number of trivalent vertices of $\Gamma$ and recall that we denoted $l: W\rightarrow \mathbb{N}$ the Bruhat length on $(W,S)$.
	
	\begin{prop}\label{Prop:higher-lam-weight}
		Let $\Gamma$ be a representative of a higher lamination on a triangle $t$ inducing the labels $x, y, z\in W$ on the sides, and $F_0$, $F_2$ two flags in relative position $z$. The number of flags $F_1$ such that $(F_0,F_1,F_2)$ extends to a valid configuration on $t\backslash \Gamma$ is:
		\begin{equation}
		Q^{\ram(\Gamma)}q^{1/2(l(x)+l(y)-l(z))}~.
		\end{equation}
	\end{prop}
	The proof is similar to the one for \Cref{Prop:higher-lam-existence}, with a  distinction of the same three cases from \Cref{Fig:three-cases}. Note that by \Cref{Prop:higher-lam-existence}, any contribution from a triple of flags appears at most once.
	
	\begin{proof}
		We reason by induction on the number of faces of $\Gamma$.
		The proposition is true for the empty graph.
		
		\medskip
		\underline{Case 1:} There is a simple edge next to $F_1$.
		\smallskip
		
		\noindent The restriction $\Gamma'$ of $\Gamma$ to the triangle $(F_0,F_1',F_2)$ is a representative of a higher lamination with strictly less faces. We have $\ram(\Gamma') = \ram(\Gamma)$, $z'=z$, $l(x')=l(x)-1$ and $l(y') = l(y)-1$. 
		The only restriction on $F_1$ is its relative position to $F_1'$ which is a simple reflection. If $F_1'$ is fixed this are $q$ flags satisfying this constraint. 
		Using the induction hypothesis, the number of possible $F_1$'s is: 
		\begin{equation}
		q\times Q^{\ram(\Gamma')}q^{1/2(l(x')+l(y')-l(z'))} = Q^{\ram(\Gamma)}q^{1/2(l(x)+l(y)-l(z))}~.
		\end{equation}
		
		\medskip
		\underline{Case 2:} There is a ramification point next to $F_1$.
		\smallskip
		
		\noindent Let $F_1'$ be a flag in one of the regions below $F_1$ (say on the boundary of $y$). The restriction $\Gamma'$ of $\Gamma$ to the triangle $(F_0,F_1',F_2)$ is a representative of a higher lamination with strictly less faces. We have $\ram(\Gamma') = \ram(\Gamma)-1$, $z'=z$, $x'=x$ and $l(y') = l(y)-1$. 
		If $F_1'$ is fixed the, by \Cref{Lem:flags-line} the flag $F_1''$ next to $F_1'$ is uniquely determined. The only restriction on $F_1$ is its relative position to $F_1'$ and $F_1''$. Since these are the same simple reflection, there are $q-1$ possible $F_1$'s. Using the induction hypothesis and the fact that $Q=q^{1/2}-q^{-1/2}$, the number of possible $F_1$'s is: \begin{equation}
		(q-1)\times Q^{\ram(\Gamma')}q^{1/2(l(x')+l(y')-l(z'))} = Q^{\ram(\Gamma)}q^{1/2(l(x)+l(y)-l(z))}~.
		\end{equation}
		
		\medskip
		\underline{Case 3:} There is a braid vertex next to $F_1$.
		\smallskip
		
		\noindent Let $F_1'$ be the flag assigned to the face of $\Gamma$ below $F_1$ as shown in \Cref{Fig:three-cases}. The restriction $\Gamma'$ of $\Gamma$ to the triangle $(F_0,F_1',F_2)$ is a representative of a higher lamination with strictly less faces (as follows from the same reasoning as in the proof of \Cref{Prop:higher-lam-existence}). Furthermore, all parameters $\ram(\Gamma'), l(x'), l(y')$ and $l(z')$ are the same as the ones corresponding to $\Gamma$. If we know the flag $F_1'$, by \Cref{Lem:flags-line} this fixes all boundary flags of the triangle $(F_0,F_1',F_2)$. In particular, this gives two flags in opposite regions around the braid vertex. By \Cref{Lem:flags-opp}, this determines uniquely $F_1$. We conclude by the induction hypothesis.
	\end{proof}
	
	A consequence of this last proposition is that the number of ramification points $\ram(\Gamma)$ does not depend on the representative of the higher lamination. One can easily check that all the relations of \Cref{Sec:graphic-relations} preserve the number of ramification points.

	\subsubsection{Product in the Hecke algebra}
	
	We now describe the graphical computation of a product $h_xh_y$ in the Hecke algebra, for some $x,y\in W$. 
	
	Let us choose reduced expressions for $x$ and $y$ and write them on the two upper sides of the triangle following the counterclockwise orientation of the boundary. For each triple of flags appearing in the product as described in \Cref{Prop:structure-const-flags} we choose a corresponding graph $\Gamma$. By \Cref{Prop:higher-lam-existence}, these graphs represent higher laminations.
	On the last edge of the triangle we read a reduced expression for some element $h_z(\Gamma)\in\He$.
	
	\begin{thm}\label{Thm:Hecke-graphique}
		We have 
		\begin{equation}
		h_xh_y = \sum_{\Gamma} Q^{\ram(\Gamma)} h_z(\Gamma)~,
		\end{equation}
		where the sum runs over all isotopy classes of graphs $\Gamma$ coming from triples of flags associated to $h_xh_y$.
		Assuming \Cref{Conj:complete-relations}, the sum can be taken over all higher laminations inducing $x$ and $y$ on the two upper sides of the triangle.
	\end{thm}
	
	\begin{proof}
		Recall the structure constants $\tensor{C}{_{xy}^{z}}$ of the Hecke algebra with quadratic relation $T_s^2=(q-1)T_s+q$. Combining \Cref{Prop:structure-const-flags} with \Cref{Prop:higher-lam-weight}, we get
		\begin{equation}
		\begin{split}
		\tensor{C}{_{xy}^{z}} &= \# \{ F_1 \in G/B \mid (F_0,F_1) = x, (F_1,F_2) = y\} \\
		&= \sum_{\Gamma_z} Q^{\ram(\Gamma_z)}q^{1/2(l(x)+l(y)-l(z))}~,
		\end{split}
		\end{equation}
		where the sum is taken over all graphs $\Gamma_z$ which induce $z$ on the third side of the triangle.
		
		Let us now relate the structure constants of $\He$ in the basis $(T_w)_{w\in W}$ to those in the standard basis $(h_w)_{w\in W}$. The two basis are linked by $T_w=q^{l(w)/2}h_w$. Hence:
		\begin{equation}
		\tensor{c}{_{xy}^{z}} = q^{-1/2(l(x)+l(y)-l(z))}\tensor{C}{_{xy}^{z}}~.
		\end{equation}
		Therefore, we get:
		\begin{equation}
		h_xh_y = \sum_{z\in W} \tensor{c}{_{xy}^{z}} h_z = \sum_{z\in W} \sum_{\Gamma_z} Q^{\ram(\Gamma_z)}h_z = \sum_\Gamma Q^{\ram(\Gamma)}h_z(\Gamma)~.
		\end{equation}
		Assuming \Cref{Conj:complete-relations} we can uniquely associate a higher lamination to each triple of flags appearing in $h_xh_y$. 
	\end{proof}
	
	Let us give a concrete example of the graphical interpretation of a product in the Hecke algebra $\He_{(\mf{S}_3,\{s,t\})}$.
	\begin{example}
		Let us multiply $h_{sts}$ with $h_{st}$ in $\He_{(\mf{S}_3,\{s,t\})}$. The direct computation reads:
		\begin{equation}
		\begin{split}
		h_{sts}h_{st} &= h_sh_th_s^2h_t \\
		&= h_sh_t^2 + Q h_sh_th_sh_t \\
		&= h_s + Qh_sh_t + Q h_s^2h_th_s\\
		&= h_s + Q h_{st} + Q h_{ts} + Q^2 h_{sts}~,
		\end{split}
		\end{equation}
		and it corresponds to the graphs of \Cref{Fig:graphicalcalculus}.
		\begin{figure}[h!]
			\centering
			\includegraphics[width=0.8\textwidth, trim=0 100 0 0, clip]{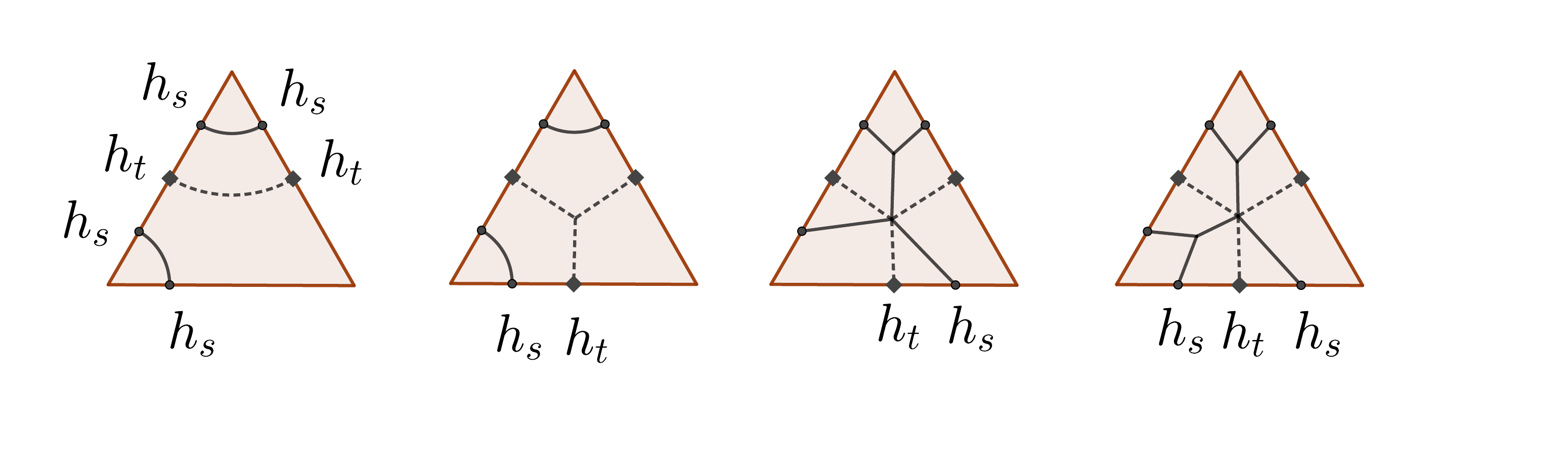}
			\caption{The graphical analogue of the product $h_{sts}h_{st}$.}\label{Fig:graphicalcalculus}
		\end{figure}
	\end{example}

	\begin{Remark}\label{Rem:ramified-cover}
		To a higher lamination $\Gamma$ on a triangle $t$ with set $R$ of ramification points, one can associate a monodromy map $\pi_1(t\backslash R) \to W$ in the following way. To any based loop $\gamma \in \pi_1(t\backslash R)$ which intersects $\Gamma$ transversely one associates the product of all labels of $\gamma \cap \Gamma$ following the orientation of $\gamma$. It is easy to check that it only depends on the homotopy class of $\gamma$.
		
		For higher lamination in type $A_n$, it is then possible to associate to $\Gamma$ an $n$-sheeted cover with simple ramification points at $R$ and a trivialization over each connected component of $t\backslash \Gamma$, such that the transition between two adjacent regions is given by the label on the separating edge.
		This is why we call the trivalent vertices of $\Gamma$ ramification points.
	\end{Remark}

	\subsection{Higher laminations for ciliated surfaces}
	
	We now define higher laminations on surfaces, which provides a diagrammatic way to compute our polynomial invariant. This viewpoint gives a direct proof of the invariance under $q \mapsto q^{-1}$ for a closed surface (this is \Cref{Prop:invariance-poly} below).
	
	Let us consider a ciliated surface $\Sigma$ where each boundary component is labeled by an element of the Weyl group, and fix a triangulation $\mathsf{T}$ of $\Sigma$.
	A \textit{higher lamination} $\Gamma$ on $\Sigma$ is the gluing of representatives of higher laminations on all triangles of $\mathsf{T}$ (their boundary data has of course to coincide).
	
	\begin{Remark}
		It should be possible to define higher laminations without using a triangulation as an equivalence class of minimal labeled graphs inducing the elements of $W$ on the boundary of $\Sigma$, mimicking the definition for triangles. The non-trivial point is to define minimality.
	\end{Remark}
	
	Let us see how to compute our polynomial using the graphical calculus. Draw all possible higher laminations on $\Sigma$ compatible with the boundary data. Recall that $\ram(\Gamma)$ is the number of trivalent vertices of $\Gamma$. As in \Cref{Thm:Hecke-graphique} we fix representatives of the higher laminations associated to triples of flags. Assuming \Cref{Conj:complete-relations}, this choice is irrelevant.
	
	\begin{thm}\label{Thm:poly-via-laminations}
		The polynomial invariant for a ciliated surface $\Sigma$ and a Weyl group $W$ is given by 
		\begin{equation}
		P_{\Sigma,W}(Q) = \sum_\Gamma Q^{\ram(\Gamma)}~,
		\end{equation}
		where the sum runs over all higher laminations of type $W$.
	\end{thm}
	\begin{proof}
		Let us consider a triangulation $\mathsf{T}$ of $\Sigma$ and a higher lamination $\Gamma$. The latter allows to associate elements of the standard basis $h_w$ to each edge of $\mathsf{T}$. By \Cref{Thm:Hecke-graphique} we know that for a triangle $t$ the contribution of $\Gamma$ to the structure constant associated to $t$ is given by $Q^{m_t}$ where $m_t$ is the number of ramification points in the triangle. 
		We conclude by the definition of the polynomial invariant:
		\begin{equation}
		\begin{split}
		P_{\Sigma,W}(Q) &= \sum_e \prod_t c_{xyz}(Q) = \sum_e \prod_t \sum_{\Gamma_e} Q^{\ram(\Gamma|_t )} \\
		&= \sum_e \sum_{\Gamma_e} Q^{\ram(\Gamma)} \\
		&= \sum_\Gamma Q^{\ram(\Gamma)}~,
		\end{split}
		\end{equation}
		where $\textstyle\sum_e$ is the sum over all possible labels of the edges of $\mathsf{T}$ by elements of $W$, $\textstyle\prod_t$ is the product over all faces $t$ of $\mathsf{T}$ and $\textstyle\prod_{\Gamma_e}$ is the product over all higher laminations compatible with the labels on the edges of $\mathsf{T}$.
	\end{proof}
	
	Let us see how this works in a simple case:
	\begin{example}
		Let us consider the sphere with three holes $\Sigma_{0,3}$ and the Hecke algebra associated with the Coxeter group $\mf{S}_2$.
		The surface $\Sigma_{0,3}$ is drawn as a triangle in the plane (together with a point at infinity one gets the sphere). We look for all possible higher laminations. 
		
		The following pictures are possible:
		\begin{center}
			\includegraphics[height=5cm]{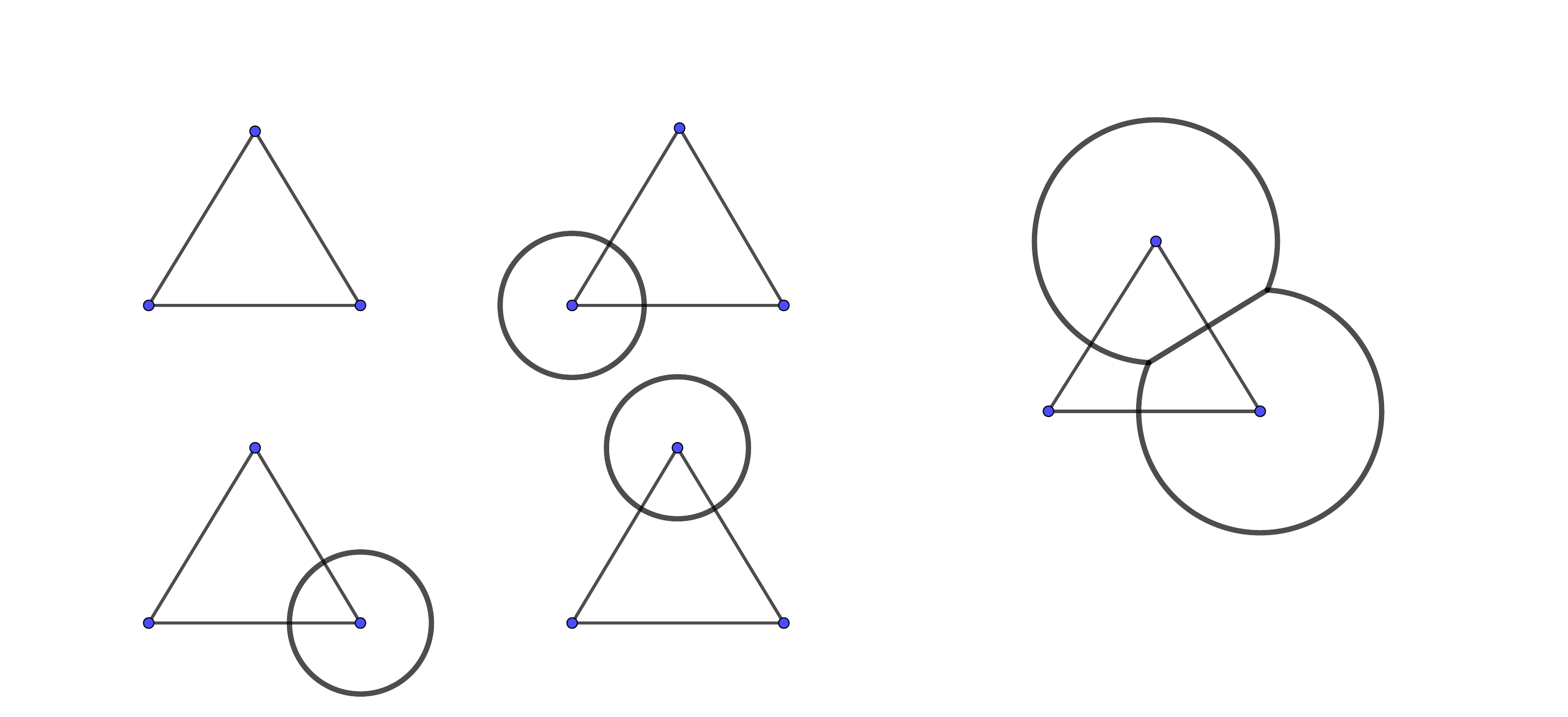}
		\end{center}
		Note that the higher lamination which goes around two vertices of the triangle is the same as the circle around the third vertex, since we are on a sphere.
		
		This shows that our polynomial is given by $4+Q^2 = q+2+q^{-1}$.
	\end{example}
	
	\begin{prop}\label{Prop:invariance-poly}
		For a punctured surface $\Sigma_{g,k}$, the invariant $P_{g,k,W}$ is a polynomial in $q=v^{-2}$. Furthermore, it is invariant under the transformation $q\mapsto q^{-1}$.
	\end{prop}
	\begin{proof}
		Using the graphical calculus, we have seen in \Cref{Thm:poly-via-laminations} that 
		\begin{equation}
		P_{g,k,W}(Q) = \sum_\Gamma Q^{\ram(\Gamma)}~,
		\end{equation}
		where the sum runs over all higher laminations.
		
		For a punctured surface, the only vertices of odd degree in a higher lamination are the ramification points. Hence there is an even number of them. Therefore our polynomial is given by a polynomial expression in $Q^2 = q-2+q^{-1}$ which is both a polynomial in $q$ and invariant under $q\mapsto q^{-1}$.
	\end{proof}

	Using the graphical calculus, we can compute the first example of an invariant for a Hecke algebra corresponding to a Coxeter system other than $\mf{S}_2$.
	\begin{example}
		Let us consider $\Sigma_{0,3}$ with $\mf{S}_3$.
		
		If we use only one color, then we are reduced to the case $\mf{S}_2$, so the contributions from one color higher laminations to our polynomial are given by $2(4+Q^2)-1$ (we have to subtract one in order to count the empty higher lamination only once).
		
		Using both colors, here are the possible pictures with their contributions:
		\begin{center}
			\includegraphics[height=6.5cm]{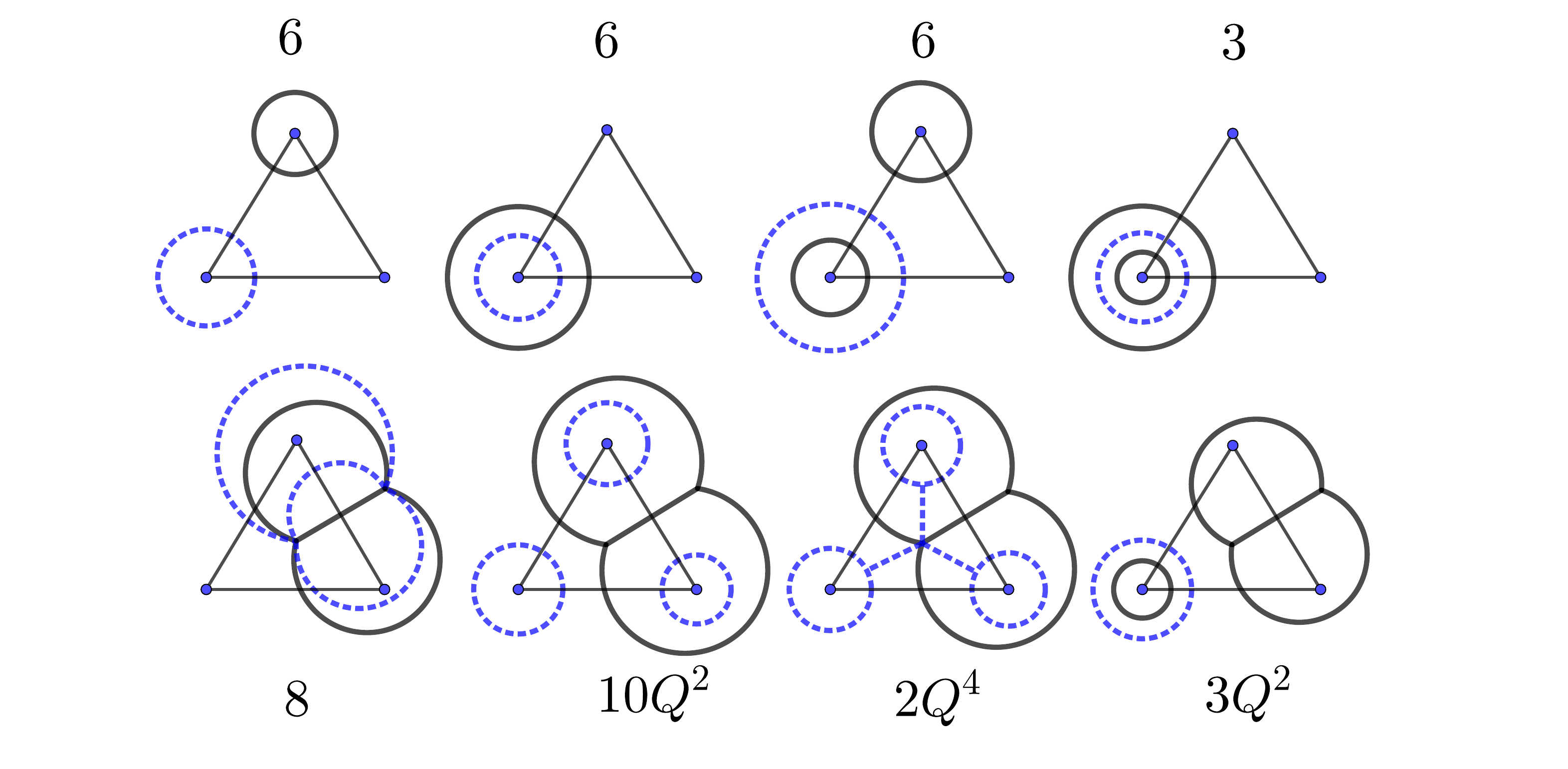}
			
			\vspace{-0.3cm}
			\includegraphics[height=4.6cm, trim= 0 0 0 60, clip]{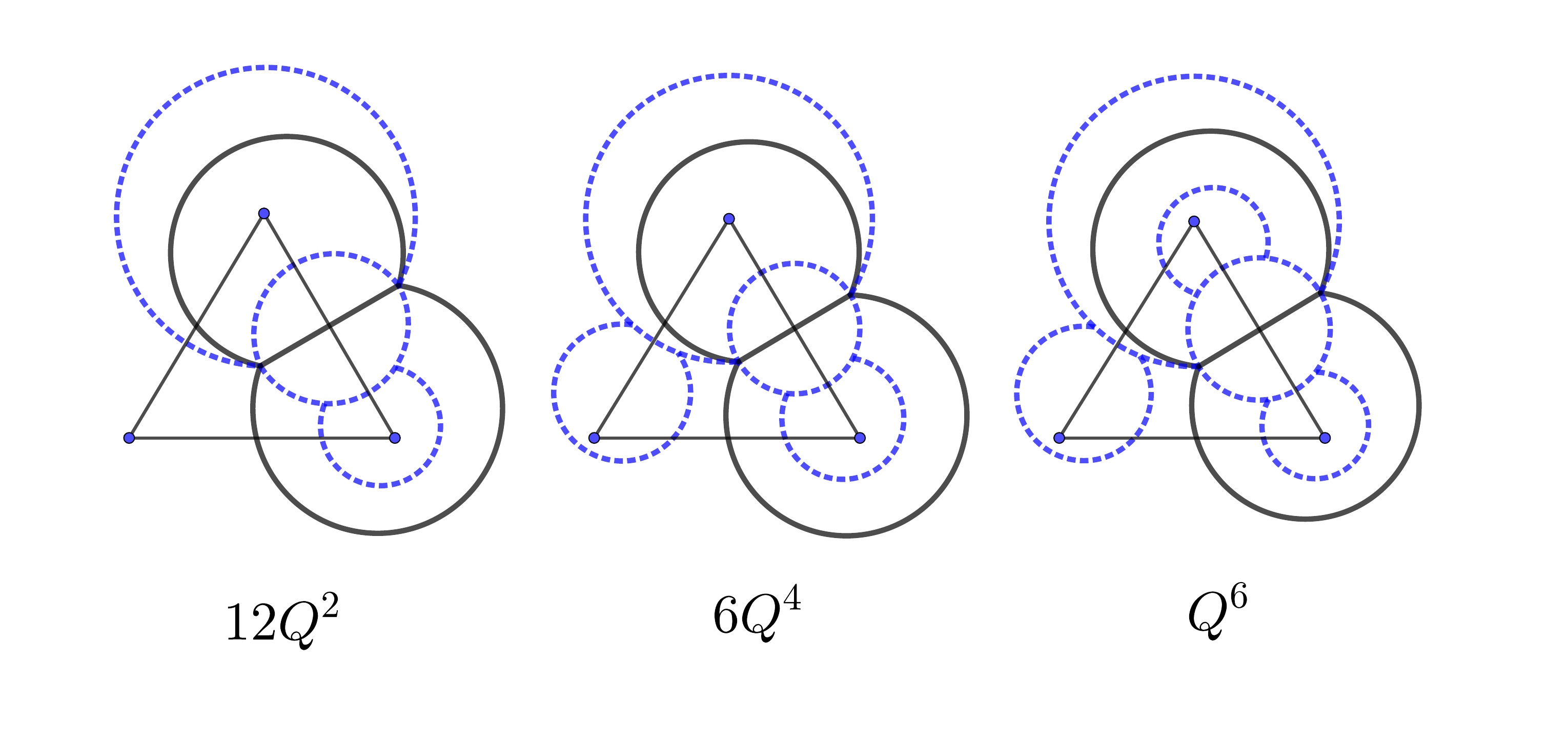}
		\end{center}
		
		The contribution of a picture is given by $Q$ to the power the number of ramification points, times the multiplicity. For example, the fifth picture has no ramification point and multiplicity 8 since for each of the three dashed loops there are two possibilities (going around a vertex of the triangle or not). For the sixth picture the dashed circles can be there or not with the only requirement that there must be at least one otherwise it is a one-color diagram, hence there are $8-1=7$ possibilities. Then colors can be exchanged, but this gives only 3 new possibilities.
		
		Adding up all the terms, we get $$36+27Q^2+8Q^4+Q^6 = q^3+2q^2+10q+10+10q^{-1}+2q^{-2}+q^{-3}.$$
		You see the weakness of the graphical calculus: there is no easy way to check whether all the possible higher laminations have been found.
	\end{example}
	
	\begin{Remark}
		In \Cref{Rem:ramified-cover}, we have seen that higher laminations are linked to ramified covers of special type (some marking and a minimality condition). Our polynomial counts these ramified covers.
	\end{Remark}

	

	

	\section{Hecke topological quantum field theory}\label{Sec:TQFT}
	
	In this section we construct the invariants introduced above in a more intrinsic way. Given a Coxeter system $(W,S)$, we construct a $2$-dimensional topological quantum field theory which associates a copy of the Hecke algebra $\He$ or its dual $\He^*$ to topological segments, while ciliated surfaces play the role of cobordisms. Punctured surfaces define elements of the base ring $\mathbb{Z}[v^{\pm1}]$ which are nothing else than the invariants of \Cref{Sec:invariant}.

	\subsection{Triangle invariants and gluing revisited}\label{Sec:trianginvrevisited}
	
	We redefine our construction in more intrinsic terms, showing in particular its independence from any choice of basis in $\He$ and to allow boundary labels to be any elements in $\He$ and not only those of the standard basis.
	
	First, we notice that the non-degenerate pairing given by the trace gives a canonical isomorphism $i$ between $\He$ and $\He^*$. In the standard basis this is given by
		\begin{equation}\label{Eq:isomorphism}
	i: \textstyle\sum_{w\in W} c_w h_w  \mapsto  \textstyle\sum_{w\in W} c_w h_{(w^{-1})}~.   
	\end{equation}
	The isomorphism $i$ is independent of a basis. To see this, take any basis $(A_w)_{w\in W}$ of the $A$-module $\He$. Denote by $(A^w)$ the trace-dual basis, i.e. such that $\tr (A_xA^y) = \delta_x^y$ for all $x,y \in W$. Then 
	\begin{equation}\label{Eq:isomorphism2}
	i: \textstyle\sum_{w\in W} c_w A_w  \mapsto  \textstyle\sum_{w\in W} c_w A^w~ 
	\end{equation}
	which is in accordance with \Cref{Eq:isomorphism} since $h^w = h_{w^{-1}}$.
	
	Second, we can redefine our construction for a triangle, with oriented edges. Choose an orientation of the triangle. To the triangle, we associate the tensor $c \in \He^a\otimes \He^b\otimes \He^c$ which comes from the multiplication in $\He$. Here, $a, b, c \in \{1,*\}$ depending on whether the orientation of the edge is in accordance with the orientation of the triangle or not. Whenever the orientation of an edge is not consistent with the orientation of the triangle, we can use the isomorphism $i$. 
	For example, in \Cref{Fig:trianginv} we get a tensor $c\in \He\otimes\He\otimes \He^*$ which to $h, h'\in \He$ and $g\in \He^*$ associates
$$c(h,h',g) = g(hh') = \tr(hh'i(g))~.$$
Since $c_{xyz}=\tr(h_xh_yh_z)$ we get the link to our initial construction in \Cref{Sec:struccons}.

	\begin{figure}[h!]
		\centering
		\includegraphics[scale=0.8]{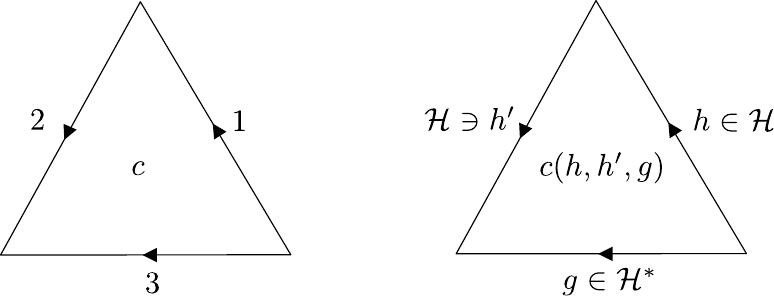}
		\caption{Assigning a tensor to triangles.}\label{Fig:trianginv}
	\end{figure}
	
	Finally, when we glue two triangles along an edge with opposite orientations, we use the natural pairing between $\He^*$ and $\He$. 
	Consider for example the two triangles drawn on the left of \Cref{Fig:gluing}. The upper-one corresponds to the tensor $c_u = c \in \He^*_1\otimes\He^*_2\otimes\He^*_3$ which is the one of \Cref{Eq:metricintui}, and the lower-one, to the corresponding tensor $c_d \in \He_4\otimes\He^*_5\otimes\He^*_6$. The indices tell to which edge the copies of $\He$ or $\He^*$ correspond. The edge $3$ is positively oriented with respect to the upper triangle hence is associated a copy of $\He$ while the edge $4$ is negatively oriented with respect to the lower one and is associated a copy of $\He^*$. Hence we can glue the edges $3$ and $4$ together. 
	
	\begin{figure}[h!]
		\centering
		\includegraphics[scale=0.7]{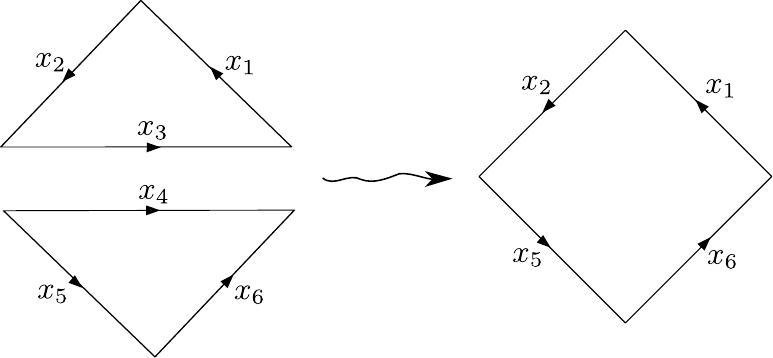}
		\caption{Gluing two triangles along an edge.}\label{Fig:gluing}
	\end{figure}

	It is clear that the gluing does not depend on the orientation of the edge (since the isomorphism $i$ can be used to change this orientation). Since our redefinition is a linear extention of the construction of \Cref{Sec:invpuncsur}, the gluing does not depend on which diagonal of the quadrilateral we choose (by \Cref{Thm:trianginv}).

	One can glue arbitrary ciliated surfaces together along edges in the same way. Since any ciliated surface of our interest admits a triangulation, it is assigned a tensor invariant by gluing the invariants corresponding to the faces of its triangulation. This tensor does not depend on the triangulation since any two triangulations can be related by a finite sequence of flips, hence it is a topological invariant of the ciliated surface.

	\subsection{Hecke topological quantum field theories}
	
	First we describe our category $\mathcal{C}$ of cobordisms: the objects are disjoint unions of oriented segments, and the morphisms from an object $A$ to another object $B$ are the ciliated surfaces whose boundary is the disjoint union of $B$ and $A$ with the orientation reversed.
	
	Let us consider the functor 
	\begin{equation}
	F:\mathcal{C}\rightarrow~ \mathbb{Z}[v^{\pm 1}]-\mathrm{Mod}
	\end{equation}
	which associates a copy of $\He$ to the positively oriented segment, a copy of $\He^*$ to the negatively oriented segment, and the proper tensor product of copies of $\He$ and $\He^*$ to a disjoint union of oriented segments. For $A$ and $B$ two objects in $\mathcal{C}$ and a ciliated surface $\S\in\mathrm{Hom}(A,B)$, the morphism $F(\S)$ is the tensor invariant constructed in \Cref{Sec:trianginvrevisited} which is indeed in $F(\partial \S)$. By \Cref{Thm:trianginv}, we know that this tensor only depends on $\Sigma$ and not on a triangulation.
	
	\begin{thm}\label{Thm:Hecke-TQFT}
		For any finite Coxeter system $(W,S)$ the functor $F$ satisfies the axioms of a topological quantum field theory listed in \cite{atiyah1988topological}.
	\end{thm}
	
	\begin{proof}
		Our construction makes clear that $F$ is invariant under orientation preserving diffeomorphisms of ciliated surfaces and their boundary, that it is involutory, and that it is multiplicative.
	\end{proof}
	
	Let $\Sigma$ be a ciliated surface without cilia and hence without boundary. Then $F(\S)\in\mathbb{Z}[v^{\pm1}]$, that is, it a Laurent polynomial in $v$. It clearly coincides with the invariant for punctured surfaces defined in \Cref{Sec:invpuncsur}.

	\subsection{Invariants of $n$-gons}
	
	Since our construction is a TQFT, i.e. behaves well under gluing, we can decompose a ciliated surface into elementary parts. Representing a surface as the gluing of the edges of a polygon, we get simple expressions for the invariant.
	
	Consider an arbitrary basis $(C_w)_{w\in W}$ of the $A$-module $\He$ and let $(C^w)_{w\in W}$ the trace-dual basis in $\He$.	
	The tensor associated to a polygon is simply given by
	\begin{prop}\label{Prop:polygon-invariant}
	Consider a polygon with $n$ edges, all oriented in the same way. Then the tensor $c_n$ associated to the polygon is given by
	\begin{equation}
	c_n = \tr(C^{w_1}\cdots C^{w_n}) C_{w_1}\otimes \cdots \otimes C_{w_n} ~.
	\end{equation}
	\end{prop}
	Note that this is equivalent to: the tensor of a polygon associates to $(h_1, ..., h_n) \in \He^n$ (one label for each edge) the scalar $\tr(h_1\cdots h_n)\in \Z[v^{\pm 1}]$.
	
	\begin{proof}
	We use induction on $n$. For $n=3$ the formula is true by the definition of our invariant. For the passage from $n$ to $n+1$, decompose a $(n+1)$-gon into one triangle and an $n$-gon (see \Cref{Fig:polygon-dissection}. We then get:
	\begin{align*}
	c_{n+1} &= \left\langle \tr(C^{w_1}\cdots C^{w_n'})\, C_{w_1}\otimes \cdots \otimes C_{w_n'}, \tr(C^{w_n}C^{w_{n+1}}C_{w_n'})\,C_{w_n}\otimes C_{w_{n+1}}\otimes C^{w_n'}\right\rangle \\
	&= \tr(C^{w_1}\cdots C^{w_n'})\tr(C^{w_n}C^{w_{n+1}}C_{w_n''})\, C^{w_n'}(C_{w_n'})\, C_{w_1}\otimes \cdots \otimes C_{w_n}\otimes C_{w_{n+1}} \\
	&= \tr(C^{w_1}\cdots C^{w_n}C^{w_{n+1}})\, C_{w_1}\otimes \cdots \otimes C_{w_n}\otimes C_{w_{n+1}}~, 
	\end{align*}
	where we used the induction hypothesis in the first line, the gluing formula in the second, and the contraction of indices property of tensors in the last line.
	\end{proof}
	
	\begin{figure}[h!]
	\centering
	\includegraphics[height=3.7cm, trim=0 80 0 110, clip]{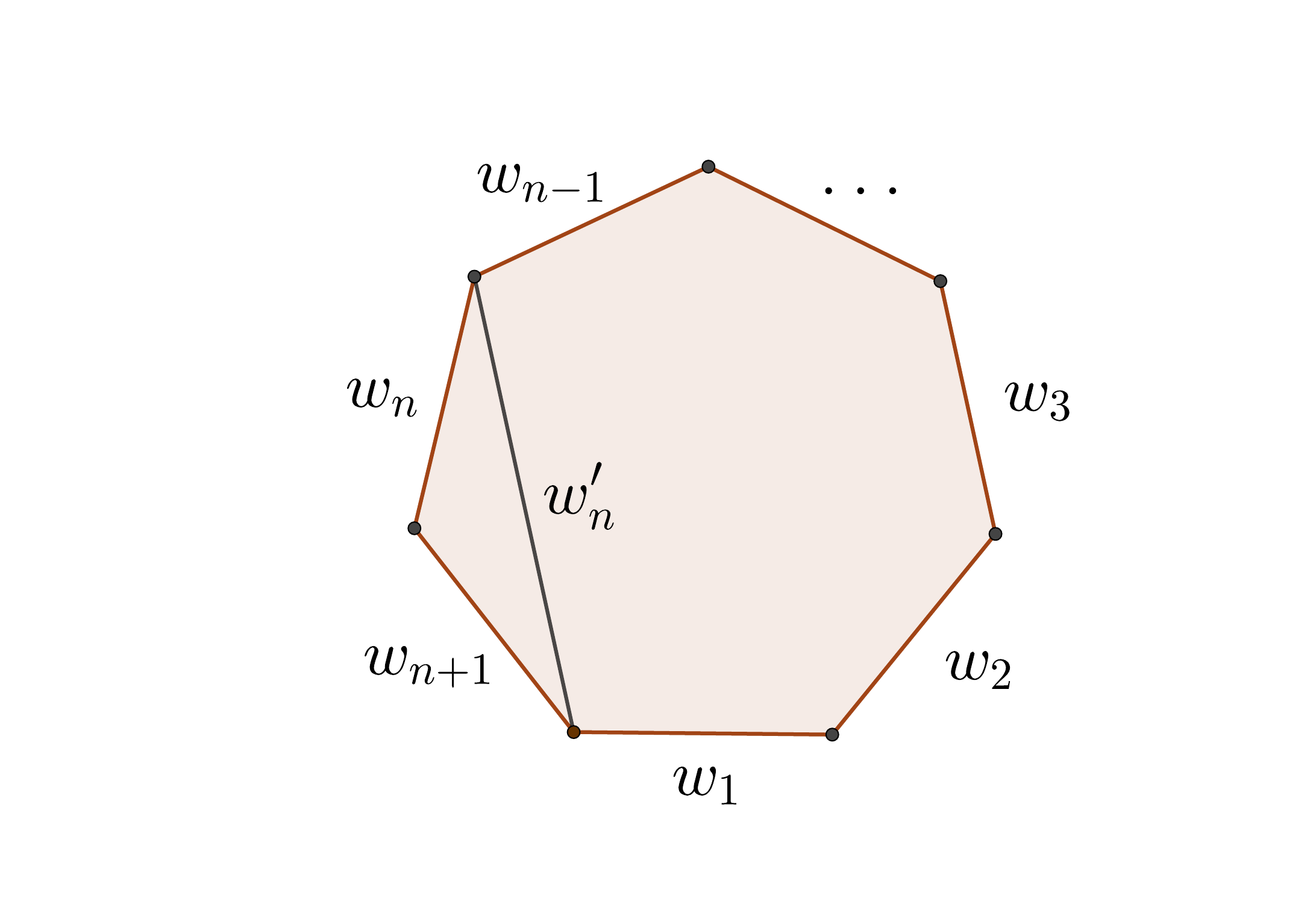}
	\caption{Decomposition of a polygon.}\label{Fig:polygon-dissection}
	\end{figure}
	
	As a consequence, we get a simple formula for the invariant of a ciliated surface without boundary components (only marked points):
	\begin{coro}
	The polynomial invariant for a surface $\S_{g,k}$ (where $k\geq 1$) is given by
	\begin{equation}\label{easyexpr}
	P_{g,k,W} = \tr \;(C_wC^w)^{k-1}(C_xC_yC^xC^y)^{g}~.
	\end{equation}
	\end{coro}
	
	\begin{proof}
	A surface $\S_{g,k}$ can be obtained as the gluing of a polygon with $4g+k-1$ edges (like in \Cref{Fig:q1-1} for $k=1$). The first $4g$ edges add handles, while the last $k-1$ edges add cones (so a marked point).
	By \Cref{Prop:polygon-invariant}, the gluing property and contraction, we get the result.
	\end{proof}
	
	\begin{Remark}
		There are lots of possible expressions, all equivalent, for a given surface $\S_{g,k}$, one for each gluing of a polygon giving $\S_{g,k}$.
	\end{Remark}

	\begin{example}\label{Ex:sphere3trace}
		Consider the sphere with three punctures, represented by the gluing of two triangles as in \Cref{Fig:sphere-reading}. Reading the picture yields a map $\bb{Z}[v^{\pm 1}] \rightarrow \bb{Z}[v^{\pm 1}]$ given by
		\begin{equation}
		1 \mapsto \tr(C^xC^yC^z)\tr(C_xC_zC_y)= \tr(C^xX^yC_yC_x)~.
		\end{equation}
		This gives the same result as representing the sphere as the gluing of a square.
	\end{example}
	\begin{figure}[h!]
		\centering
		\includegraphics[height=5.5cm, trim= 0 10 0 10, clip]{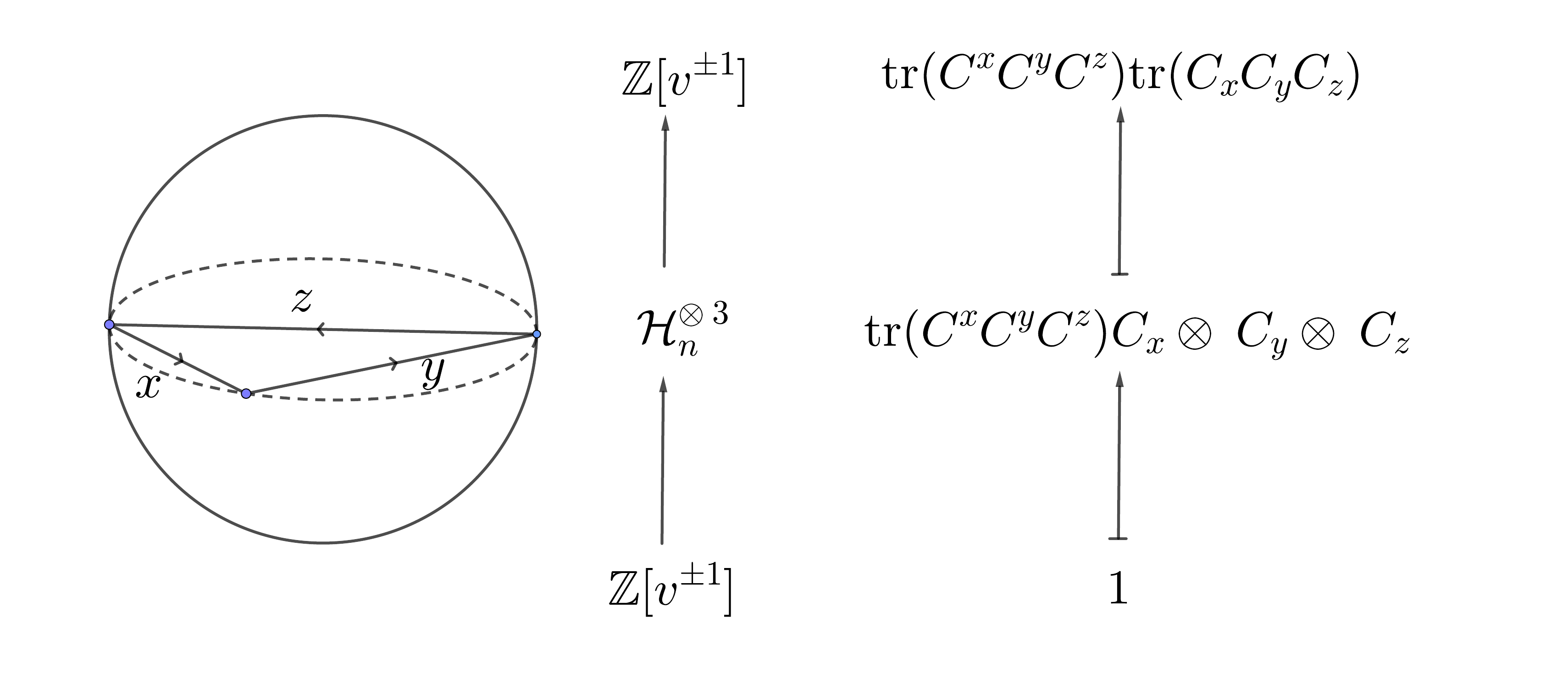}
		\caption{Reading off the polynomial}\label{Fig:sphere-reading}
	\end{figure}

	\section{Explicit expression and positivity}
	
	In this section, we explicitly compute the polynomial invariant for ciliated surfaces and determine in which cases positivity hold.
	The key observation is that for a punctured surface, the polynomial is the trace of a central element in the Hecke algebra $\He$. The structure of the center can be understood using the Wedderburn decomposition of $\He$. The polynomial is a sum of powers of Schur elements multiplied by irreducible characters.

	\subsection{Schur elements and Wedderburn decomposition}
	
	The Hecke algebra with its standard trace is a specific example of a symmetric algebra, that is, an algebra with a non-degenerate trace function. There is a general theory of these algebras which we present briefly here. In particular we expose some general facts about Schur elements and the Wedderburn decomposition for Hecke algebras. Main references are \cite{geck2000characters}, Chapter 7 and 8 of \cite{Maria-HDR} and the article \cite{neunhoffer2006kazhdan}.
	
\paragraph{Symmetric algebras.}
Fix a commutative integral domain $A$. 
A \textbf{symmetric algebra} $H$ is an $A$-algebra, which is free and finitely generated over $A$, equipped with a symmetric, non-degenerate trace $\tau: H \rightarrow A$. Symmetric means that $\tau(hh') = \tau(h'h) \;\forall h, h' \in H$.
The trace gives an isomorphism $H \cong H^* = \Hom_A(H,A)$ via $h\in H \mapsto \tau(h\; .\;)$. Denote by $B$ the associated quadratic form $B(h_1,h_2) = \tau(h_1h_2)$.

Denote by $T(H) \simeq (H/[H,H])^*$ the space of traces on $H$. 
Note that the space of traces $T(H)$ is canonically isomorphic to the center $Z(H)$ of $H$ (see for example \cite[Proposition 7.1.7]{geck2000characters}).

We further assume $H$ semisimple. Note that if $A=\R$ and $B$ is positive definite, the algebra $H$ is automatically semisimple since an orthogonal complement to a left ideal of $H$ is a right ideal and \textit{vice versa}.

Denote by $\mathrm{Irr}(H)$ the set of irreducible representations of $H$ and for $\lambda\in \mathrm{Irr}(H)$ denote by the same letter the map $\lambda : H\to \mathrm{End(V_\lambda)}$. Then the semisimplicity of $H$ implies that
\begin{equation}\label{Eq:Wedderburn}
    H = \bigoplus_{\lambda \in \mathrm{Irr}(H)} \mathrm{End(V_\lambda)} ~.
\end{equation} 

This is the Artin--Wedderburn decomposition, or a version of the Peter--Weyl theorem in this setting. 
	
For any representation $\lambda$ of $H$ one can associate a trace called a character denoted by $\chi_\lambda\in T(H)$, defined by $\chi_\lambda(h)=\mathrm{tr} \lambda(h)$ and a corresponding element $Z_\lambda$ of the center $Z(H)$. The sets $\{\chi_\lambda\}$ and $\{Z_\lambda\}$ for $\lambda\in \mathrm{Irr}(H)$ form orthogonal bases in the spaces $T(H)$ and $Z(H)$, respectively. Denote by  $s_\lambda$ the inverses the coefficients of the decomposition of unity in $H$ with respect to the base of the $Z_\l$:
\begin{equation}\label{Eq:Schur-def-2}
    \textstyle\sum_\l \frac{1}{s_\lambda}Z_\lambda = 1 ~.
\end{equation}
The $s_\l$ are called \textbf{Schur elements}. From \Cref{Eq:Schur-def-2} we immediately get
\begin{equation}
    \tau = \textstyle\sum_\l \frac{1}{s_\l}\chi_\l ~.
\end{equation}
Further, by \Cref{Eq:Wedderburn}, we see that $Z_\l$ acts on $V_\l$ as $s_\l \id$ (which is the usual definition of the Schur elements).
The elements $Z_\l$ can be computed as follows:
\begin{prop}\label{Prop:Zlambda}
Let $(C_w)_{w\in W}$ be a basis of $H$ and let $(C^w)$ be its trace-dual basis in $H$. The central element $Z_\l$ can be computed by
    \begin{equation}
        	Z_\l = \sum_{w\in W} \chi_\l(C_w)C^w~.
    \end{equation}
\end{prop}
	\begin{proof}
	    For $h\in H$, we compute (using Einstein summation convention):
	    \begin{equation}
	        \tau(Z_\l h) = \chi_\l(h) = \chi_\l(\tau(C^wh)C_w) = \tau(C^wh)\chi_\l(C_w) = \tau(\chi_\l(C_w)C^wh)
	    \end{equation}
	    where we used that $h=\tau(C^wh)C_w$ by definition of $C^w$.
	    Hence the proposition follows by the non-degeneracy of $\tau$.
	\end{proof}

	There is another definition of Schur elements which will be useful in the sequel. Let $\varphi: V\rightarrow V'$ be an $A$-morphism between right $H$-modules. Let $I(\varphi): V\rightarrow V'$ be defined by:
	\begin{equation}\label{def-I}
	I(\varphi)\cdot v = \sum_w \varphi(vC_w)C^w~.
	\end{equation}
	The morphism $I(\varphi)$ does not depend on the choice of basis $(C_w)$ and it is a morphism of $H$-modules \cite[Lemma 7.1.10]{geck2000characters}.
	
	For $V=V'=V_\l$ an irreducible representation, we have \cite[Theorem 7.2.1]{geck2000characters}:
	\begin{prop}\label{Prop:Schur-I}
		Let $\varphi\in \End(V_\l)$. Then:
		\begin{equation}
		I(\varphi) = s_\l \tr(\varphi) \id~.
		\end{equation}
	\end{prop}

	\paragraph{Hecke algebras.}
	Now, we specialize to the Hecke algebra $(\He, \tr)$ with its standard trace. By \cite{gyoja1989semisimplicity}, the Hecke algebra is semisimple over the localized ring $A=\mathbb{Z}[q^{\pm 1}]/P(q)$ where $P$ is the Poincar\'e polynomial of $\He$.
	The Artin--Wedderburn theorem implies that 
	\begin{equation}\label{Eq:AW}
	\He \simeq \bigoplus_{\l \in \Irr(\He)} \End(V_\l)~.
	\end{equation}
	
	\begin{Remark}
		For Hecke algebras, the irreducible representations are all inside so-called left cell representations. For type $A$, the left cell representations are all irreducible. The article \cite{neunhoffer2006kazhdan} by Neunhöffer describes explicitly the adapted basis for type $A$, i.e. the matrix elements for each factor $\End(V_\l)$. We will use these cell representations only for \Cref{Prop:pos-characters} below.
	\end{Remark}

	The decomposition of \Cref{Eq:AW} implies that the center of the Hecke algebra is given by diagonal matrices.
	
	\begin{prop}\label{Prop:mult-center}
		The elements $(Z_\l)_{\l \in \Irr(\He)}$ form a basis of the center $Z(\He)$ such that:
		\begin{equation}
		Z_\l Z_\mu = \delta_{\l,\mu}s_\l Z_\l \;\; \forall \;\lambda, \mu\in \Irr(\He)~.
		\end{equation}
	\end{prop}
	
	\begin{proof}
		Since the characters form a basis of the space of trace functions, the $Z_\l$'s form a basis of the center $Z(H)$. The Wedderburn decomposition implies that $Z_\l Z_\mu = 0$ for $\l\neq \mu$.
		Now, for all $h\in H$:
		\begin{equation}
		\tr(Z_\l^2h) = \chi_\l(Z_\l h) = \chi_\l(s_\l h) = \tr(s_\l Z_\l h)~.
		\end{equation}
		Hence $Z_\l^2 = s_\l Z_\l$.
	\end{proof}
	
	There is a special symmetry in the Schur elements of Hecke algebras. Let $\gamma$ be the $\bb{Z}[v^{\pm 1}]$-algebra homomorphism on $\He$ given by $\gamma(h_s) = -qh_s^{-1}$. For $\l \in \Irr(\He)$ let $\l^*$ be the composition $\l\circ\gamma$ called the \textit{dual representation}. Proposition 9.4.3 in \cite{geck2000characters} states that:
	\begin{prop}\label{Prop:duality-schur-elements}
		For $\l\in\Irr(\He)$, we have $$s_{\l^*}(q) = s_\l(q^{-1}).$$
	\end{prop}

	\subsection{Central elements}
	
	Recall the expression of the invariant for punctured surfaces in terms of the standard trace of \Cref{easyexpr}:
	$$ P_{g,k,W} = \tr \;(C_wC^w)^{k-1}(C_xC_yC^xC^y)^{g}.$$
	
	For surfaces with boundary labeled by element in $\He$, terms of the form $C_whC^w$ arise.
	
	\begin{thm}
		The elements of the form $(C_wC^w)^{k-1}(C_xC_yC^xC^y)^{g}$ and $C_whC^w$ (for $h\in\He$) are in the center of the Hecke algebra.
	\end{thm}
	
	\begin{proof}
		For the first statement, let $s=(C_wC^w)^{k-1}(C_xC_yC^xC^y)^{g}$.
		By the non-degeneracy of the trace, it is sufficient to show that 
		\begin{equation}\label{trace-egal}
		\tr(sh_1h_2) = \tr(h_1sh_2) \;\forall \; h_1, h_2\in \mc{H}.
		\end{equation}
		The expression $\tr(sh_1h_2)$ is the invariant of a surface with two boundary components labeled by $h_1$ and $h_2$. The same holds for $\tr(h_1sh_2)=\tr(sh_2h_1)$ with $h_1$ and $h_2$ exchanged.
		
		A $\pi$-rotation of the surface as shown in \Cref{Fig:rotation} exchanges $h_1$ and $h_2$. Since it does not change the topology, the invariants of the two surfaces coincide. This implies \Cref{trace-egal} and thus the theorem.
		
		The second statement is analogous. The surface corresponding to the invariant 
		\begin{equation}
		    \tr C_whC^wh_1h_2
		\end{equation}
		is a cylinder with one boundary labeled by $h$ and another boundary with two cilia and labels $h_1$ and $h_2$. Again a $\pi$-rotation exchanges $h_1$ and $h_2$.
	\end{proof}
	\begin{figure}[h!]
		\centering
		\includegraphics[height=3cm]{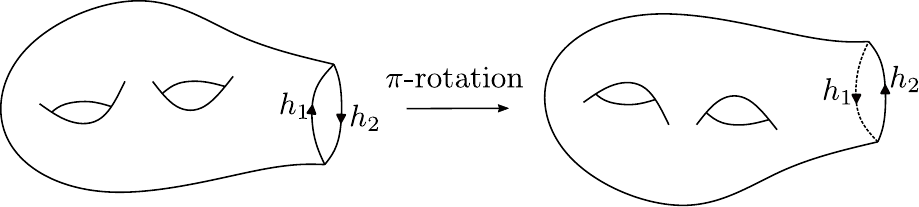}
		
		\caption{Exchanging $h_1$ and $h_2$.}\label{Fig:rotation}
	\end{figure}
	
	Note that the central elements $C_wC^w$ and $C_xC_yC^xC^y$ can be seen as Casimir elements of order 2 and 4 for the Hecke algebra.
	
	From \Cref{Prop:mult-center} we know that the Schur elements $(Z_\l)_{\l\in \Irr(\He)}$ form a basis of the center $Z(\He)$ of the Hecke algebra. We can determine the decomposition of the two building blocks $C_whC^w$ and $C_xC_yC^xC^y$ of expressions in the form of \Cref{easyexpr} in this basis:
	\begin{prop}\label{Prop:expr-in-center}
		For $h\in \He$, we have:
		\begin{equation}
		C_whC^w = \textstyle\sum_{\l} \chi_\l(h) Z_\l~,
		\end{equation}
		and 
		\begin{equation}
		C_xC_yC^xC^y = \textstyle\sum_{\l} s_\l Z_\l~.
		\end{equation}
	\end{prop}
	Note that in particular (for $h=1$), we get
		\begin{equation}
		C_wC^w = \textstyle\sum_{\l} \dim(V_\l) Z_\l~.
		\end{equation}
	\begin{proof}
		For the first assertion, take $V=V'=V_\l$ and $\varphi_h \in \End(V_\l)$, the action induced by right multiplication by $h$. Let $h'\in V_\l$ and write $C_xhC^x = \sum_\mu e_\mu Z_\mu$.
		On the one hand we have 
		\begin{equation}
		h'C_xhC^x = I(\varphi_h).h' = s_\l \tr(\varphi_h) h' = s_\l \chi_\l(h) h'
		\end{equation}
		thanks to \Cref{Prop:Schur-I}, and on the other
		\begin{equation}
		h'C_xhC^x = \textstyle\sum_\mu e_\mu Z_\mu h' = e_\l s_\l h'
		\end{equation}
		since $Z_\l h' = s_\l h'$ and 0 for other values of $\mu$.
		Comparing coefficients we get $e_\l = \chi_\l(h)$.
		
		\medskip
		For the second assertion, consider the map $\varphi_{\l,x}:V_\l \rightarrow V_\l$ given by
		\begin{equation}
		\varphi_{\l,x}(v) = vC^x~.
		\end{equation}
		Thus: 
		\begin{equation}
		I(\varphi_{\l,x}).v = \sum_y \varphi_{\l,x}(vC_y)C^y = \sum_y vC_yC^xC^y~,
		\end{equation}
		and by \Cref{Prop:Schur-I}:
		\begin{equation}
		I(\varphi_{\l,x}).v  = s_\l\tr(\varphi_{\l,x})v = s_\l \, \chi_\l(C^x)v
		\end{equation}
		where we have used the definition of the character $\chi_\l$.
		Hence:
		\begin{equation}\label{aux32}
		\sum_y vC_yC^xC^y = s_\l \,\chi_\l(C^x)\, v \text{ for } v\in V_\l.
		\end{equation}
		
		Recall the Wedderburn decomposition 
		\begin{equation}
		\He \cong \bigoplus_\l \End(V_\l)~.
		\end{equation}
		Let $(B_x)$ be an adapted basis of $\He$ given by the matrix coefficients in the factors of the decomposition, and $(B^x)$ its trace-dual basis in $\He$. 
		Let $\l(x)$ be the unique $\l$ such that $B_x \in V_\l$.
		
		Since $B_xB_y = 0$ whenever $\l(x)\neq \l(y)$, we know that $B^x \in \Span(B_y \mid \l(y) = \l(x))$, and hence $B^xB_y = 0$ for $\l(x)\neq\l(y)$. Hence $\chi_\l(A^x) = 0$ if $\l \neq\l(x)$.
		
		\Cref{aux32} implies that:
		\begin{equation}
		\sum_y B_xB_yB^xB^y = s_{\l(x)}\chi_{\l(x)}(B^x)B_x~,
		\end{equation}
		and eventually:
		\begin{equation}
		\begin{split}
		\sum_{x,y}B_xB_yB^xB^y &= \sum_x s_{\l(x)}\chi_{\l(x)}(B^x)B_x \\
		&= \sum_\l s_\l \sum_{x \,\mid\, \l(x) = \l} \chi_\l(B^x)B_x \\
		&= \sum_\l s_\l \sum_{x} \chi_\l(B^x)B_x \\
		&= \sum_\l s_\l Z_\l
		\end{split}
		\end{equation}
		where we used that $\chi_\l(B^x) = 0$ if $\l\neq \l(x)$ and that $Z_\l = \sum_x \chi_\l(B^x)B_x$ from \Cref{Prop:Zlambda}. This concludes the proof.
	\end{proof}

	\subsection{Explicit expression}
	We are now ready to compute the invariants for ciliated surfaces. We start with punctured surfaces.
	
	\begin{thm}\label{Thm:explexpr}
		The polynomial invariant corresponding to a punctured surface is given by 
		\begin{equation}
		P_{g,k,W}(q) = \sum_\l (\dim V_\l)^k s_\l(q)^{2g-2+k}~.
		\end{equation}
	\end{thm}
	
	\begin{proof}
		We have that:
		\begin{equation}
		\begin{split}
		P_{g,k,W}(q) &= \tr \left(\left(C_wC^w\right)^{k-1}\left(C_xC_yC^xC^y\right)^g\right)  \\
		&= \tr \left(\textstyle\sum_\l \dim V_\l Z_\l\right)^{k-1}\left(\textstyle\sum_\l s_\l Z_\l\right)^g \\
		&= \tr \sum_\l (\dim V_\l)^{k-1} s_\l^{2g-2+k}Z_\l \\ 
		&= \sum_\l (\dim V_\l)^k s_\l^{2g-2+k}~,
		\end{split}
		\end{equation}
		where the first equality is \Cref{easyexpr}, the second one comes from \Cref{Prop:expr-in-center}, the third one from \Cref{Prop:mult-center}, and in the last one we used $\tr Z_\l = \chi_\l(1) = \dim V_\l$.
	\end{proof}
	
	For example, let us consider the case $W=\mf{S}_2$. There are two Schur elements in $\He_{\mf{S}_2}$, respectively $s_1=1+q$ and $s_2 = 1+q^{-1}$. Hence the invariant corresponding to a genus $g$ surface with $k$ punctures is given by
	\begin{equation}
	P_{g,k,\mf{S}_2} = (1+q)^{2g-2+k}+(1+q^{-1})^{2g-2+k}~.
	\end{equation}
	
	Looking at the explicit expression in \Cref{Thm:explexpr}, we observe several phenomena:
	
	\begin{itemize}
		\item We can evaluate the expression at $k=0$, although we have no definition of our invariant for closed surfaces. In view of \Cref{Rem:ramified-cover} this is not surprising since our polynomial seems to count special ramified coverings over $\S$.
		\item 	Specializing to $q=1$ and using $s_\l(1) = \vert W\vert (\dim V_\l)^{-1}$ (in type $A$ this is the \textit{hook length formula}), we recover the result of \Cref{poly-q1}.
		\item The invariance under $q\mapsto q^{-1}$ of our polynomial can be obtained from the explicit expression of \Cref{Thm:explexpr} and the duality property of Schur elements of \Cref{Prop:duality-schur-elements}.
		Interestingly, the coefficients of each Schur element are symmetric as can be observed by a case-by-case study (see \Cref{Prop:schur-pos-sym}). It would be nice to recover this property using the TQFTs developed in this paper. 
	\end{itemize}

	Let us turn to the case of a ciliated surface $\Sigma_{g,k,\{p_1,...,p_n\}}$ of genus $g$, $k \geq 1$ punctures and boundary components with $p_i \geq 1$ cilia. Let $h_i \in \He$ be the product of the elements along the $i$-th boundary component following the orientation of $\Sigma$ ($h_i$ \textit{per se} is not well-defined but its conjugacy class is and it is enough to write \Cref{Thm:explexpr-boundary} unambiguously since characters are class functions).
	
	\begin{thm}\label{Thm:explexpr-boundary}
		For a ciliated surface $\Sigma_{g,k,\{p_1,...,p_n\}}$ with labels $(h_1,...,h_n) \in \He^n$, the polynomial invariant is:
		\begin{equation}
		P_{\S,W}(q) = \sum_\l (\dim V_\l)^{k} s_\l(q)^{2g-2+k+n}\chi_\l(h_1)\cdots \chi_\l(h_n)~.
		\end{equation}
	\end{thm}
	
	\begin{proof}
		We can suppose $n\geq 1$ since the punctures case was already treated in \Cref{Thm:explexpr}. There is a polygonal gluing yielding the ciliated surface $\S$ in the following way: start with a disc with one cilia on the boundary. Then glue $g$ handles to it, giving a term $(C_xC_yC^xC^y)^g$, then add $k$ punctures, giving a term $(C_wC^w)^k$, add other boundary components with labels $h_2$ to $h_n$, giving a term $\prod_{i=2}^n C_wh_iC^w$, and finally add the label $h_1$ to the initial boundary circle.
		
		Hence, we get:
		\begin{equation}
		\begin{split}
		P_{\S,W}(q) &= \tr \; (C_xC_yC^xC^y)^g(C_wC^w)^{k}\textstyle\prod_{i=2}^n(C_wh_iC^w)\; h_1\\
		&= \tr \;(\textstyle\sum_\l s_{\l}^{2g-1}Z_\l)(\textstyle\sum_\mu (\dim V_\mu)^ks_{\mu}^{k-1} Z_\mu)\textstyle\prod_{i=2}^n(\sum_{\l_i}\chi_{\l_i}(h_i)Z_{\l_i}) \; h_1 \\
		&= \sum_\l (\dim V_\l)^k s_\l^{2g-2+k+n}\textstyle\prod_{i=2}^n\chi_\l(h_i) \; \tr(Z_\l h_1) \\
		&= \sum_\l (\dim V_\l)^{k} s_\l^{2g-2+k+n}\textstyle\prod_{i=1}^n\chi_\l(h_i)~,
		\end{split}
		\end{equation}
		where we used \Cref{easyexpr}, \Cref{Prop:mult-center} and the equality $\tr(Z_\l h_1) = \chi_\l(h_1)$.
	\end{proof}
	
	Two remarks on the explicit expression:
	\begin{itemize}
	    \item A puncture is equivalent to a boundary component labeled by $1\in \He$, since $\dim(V_\l) = \chi_\l(1)$.
	    \item The compatibility of \Cref{Thm:explexpr-boundary} with the gluing property of \Cref{gluing-1} is ensured by the orthogonality of the irreducible characters (see \cite[Corollary 7.2.4]{geck2000characters}):
	    \begin{equation}
	        \chi_\l(C_w)\chi_\mu(C^w) = \delta_{\l,\mu}s_\l \dim(V_\l)~.
	    \end{equation}
	 \end{itemize}

	\subsection{Positivity properties}\label{Sec:final-positivity}
	We now turn to the positivity properties of our invariants. 
	For punctured surfaces this reduces to the study of positivity properties of the Schur elements. For a ciliated surface, the characters of elements in the Kazhdan--Lusztig basis appear.
	
	In \Cref{Appendix:Schur-positive} we carry out the study of positivity properties of Schur elements, using a formula of \cite{Maria-HDR}. We obtain the following theorem:
	
	\begin{thm}\label{Thm:positivity}
		The following positivity properties hold for punctures surfaces:
		\begin{enumerate}[label=(\roman*)]
			\item The polynomial invariant $P_{g,k,W}(q)$ has positive coefficients for all classical $W$ and for the exceptional types $H_3, E_6$ and $E_7$.
			\item For all other types, i.e. $I_2(m)$ for $m\geq 5$, $H_4, F_4$ and $E_8$, the invariants may have negative coefficients.
		\end{enumerate}
	\end{thm}
	
	The first part (apart from $H_3$) is a direct consequence of the positivity of Schur elements, the second is a case-by-case study. We are thankful to Sebastian Manecke for his contribution to the case $H_3$.
	
	\begin{proof}
		By \Cref{Thm:Schur-pos}, the Schur elements have positive coefficients in all classical types and for $E_6$ and $E_7$. Hence \Cref{Thm:explexpr} implies the positivity of the coefficients of the invariants for punctured surfaces in those cases.
		
		Let us now study the case $H_3$. An explicit computation shows that there are only two Schur elements with negative coefficients which only differ by a shift by $q^5$. In other words, we can write these two elements as some Laurent polynomial $P$ and $q^5P$. The corresponding irreducible representations of $\He_{H_3}$ are $3$-dimensional. There are two other Schur elements whose representation is $3$-dimensional. These are of the form $Q$ and $q^5Q$ for some other Laurent polynomial $Q$. More explicitly let $a=\cos(\tfrac{2\pi}{5})$ and $b=\cos(\tfrac{4\pi}{5})$. Then:
		\begin{equation}
		\begin{split}
		P(q) =& \; (2-2b)q+(6-6b)+(7-2b)q^{-1}-10aq^{-2} \\
		& \; -10aq^{-3}+(7-2b)q^{-4}+(6-6b)q^{-5}+(2-2b)q^{-6} \\
		Q(q) =& \; (2-2a)q+(6-6a)+(7-2a)q^{-1}-10bq^{-2} \\
		& \; -10bq^{-3}+(7-2a)q^{-4}+(6-6a)q^{-5}+(2-2a)q^{-6}~. 
		\end{split}
		\end{equation}
		We are going to show that \emph{$P^l+Q^l$ has positive coefficients for all $l$}. This implies the positivity of the coefficients of the invariants corresponding to $H_3$ by the explicit formula of \Cref{Thm:explexpr}.
		
		We note that the coefficients of $Q$ are positive, and so are those of $Q^l$ for all $l \geq 1$.
		We prove by induction $l \mapsto l+5$ that $P^l$ has positive coefficients for $l\geq 5$. This is checked by direct computation for $P^5$ to $P^9$. Then $P^{l+5} = P^5P^l$ gives the induction heredity.
		For $l\in \{1,2,3,4\}$ one checks explicitly that $P^l+Q^l$ is positive.
		
		\medskip
		To prove the second part, we note by explicit computation that in types $G_2, F_4, H_4$ and $E_8$, the polynomial has negative coefficients for $g=0$ and $k=3$.
		
		Consider now type $I_2(m)$ with $m=2l+1$ odd. Then by \cite[Thm 8.3.4]{geck2000characters}, the Schur elements are given by 
		\begin{equation}
		\begin{split}
		s_0 & = 1+2q+2q^2+...+2q^{m-1}+q^m \\
		s_{l+1} & = 1+2q^{-1}+...+2q^{1-m}+q^{-m} \\
		s_j & = \frac{m}{2-2\cos(\frac{2\pi j}{m})}(q-2\cos(\tfrac{2\pi j}{m})+q^{-1}) ~~~ \mathrm{for} ~~~ 1\leq j\leq l 
		\end{split}
		\end{equation}
		where $s_0$ and $s_{l+1}$ correspond to 1-dimensional representations and all others to representations which are of dimension 2.
		Hence the constant term of $P_{0,3,W}(q)$ is given by
		\begin{equation}
		2-8m\sum_{j=1}^l \frac{\cos(\frac{2\pi j}{m})}{1-\cos(\frac{2\pi j}{m})} = 2-8m\left(-l+\frac{1}{2}\sum_{j=1}^l \frac{1}{\sin^2(\frac{\pi j}{m})}\right)~.
		\end{equation}
		For $l\geq 5$ we have $\frac{\pi^2}{m^2}<\frac{2}{5l}$ (by analysis of the roots of a quadratic polynomial in $l$). Using $\sin(x) \leq x$ we obtain:
		\begin{equation}
		-l+\frac{1}{2}\sum_{j=1}^l \frac{1}{\sin^2(\frac{\pi j}{m})} > -l+\frac{1}{2}\frac{1}{\sin^2(\frac{\pi}{m})} > \frac{l}{4}~.
		\end{equation}
		Thus, the constant term is smaller than $2-2ml$ which is strictly negative for $l\geq 5$. For $l=3$ and $l=4$, explicit computations also show negative constant terms. For $l=2$, one checks that for $g=0, k=4$ there are negative coefficients in the polynomial.
		
		Consider type $I_2(m)$ with $m=2l$ even. Then again by \cite{geck2000characters} Theorem 8.3.4, in addition to the Schur elements given above (for $1\leq j\leq l-1$ for the 2-dimensional representations), there are two others whose corresponding representation is 1-dimensional given by:
		\begin{equation}
		s_{\varepsilon_1}=s_{\varepsilon_2} = l(q+2+q^{-1})~.
		\end{equation}
		Hence the constant term of $P_{0,3,W}(q)$ is given by 
		\begin{equation}
		2+2m-8m\left(-(l-1)+\sum_{j=1}^{l-1} \frac{1}{2\sin^2(\frac{\pi j}{m})}\right)~.
		\end{equation}
		The same techniques as for $m$ odd apply and give that this constant term is negative for $l\geq 5$. For $l=3, 4$, explicit computations show that the constant term is still negative. This concludes for the cases $I_2(m)$ with $m\geq 5$.
	\end{proof}
	
	\begin{Remark}
		Recall that the invariants are traces of elements of the form
		\begin{equation}
		(C_wC^w)^{k-1}(C_xC_yC^xC^y)^g~.
		\end{equation} 
		These expressions are independent of the choice of basis $(C_w)$, and hence we may choose the Kazhdan--Lusztig basis $(b_w)$ to do the computations. Despite all the known positivity results concerning this very special basis of the Hecke algebra, we are unable to deduce from that the positivity of our invariants because of the appearance of the dual basis $(b^w)$.
		\Cref{Thm:positivity} can be seen as a family of positivity properties in the center of the Hecke algebra indexed by punctured surfaces.
	\end{Remark}
	
	\medskip
	Let us analyze the case of a general ciliated surface with boundary components labeled by elements of the Kazhdan--Lusztig basis (KL basis for short in the sequel).
	\begin{prop}\label{Prop:pos-characters}
		In type $A$, all irreducible characters evaluated at elements of the Kazhdan--Lusztig basis are positive.
	\end{prop}
	\begin{proof}
		Fix an irreducible representation with character $\chi_\l$. In type $A$, we know that this representation is a cell representation for some left cell $\Lambda$. Further by \cite[Eq. 4]{neunhoffer2006kazhdan} , we have
		\begin{equation}
		\chi_\l(b_w) = \sum_{x\in \Lambda} \tr b_wb_xb^x~.
		\end{equation}
		Since the structure constants in the KL basis are positive, all $\tr b_xb^xb_w$ are positive.
	\end{proof}

	Let $\He_{\geq 0}$ denote the set of elements in the Hecke algebra which have non-negative coordinates in the KL basis.
	
	\begin{coro}\label{Coro:A-positive}
		In type $A$, the invariant corresponding to a ciliated surface with boundary labels in $\He_{\geq 0}$, has positive coefficients.
	\end{coro}
	\begin{proof}
		Let $\alpha$ be the product of the labels on the boundary following the latter according to the orientation induced by the surface. Since $\He_{\geq 0}$ is stable under product, we get $\alpha\in \He_{\geq 0}$. In type $A$ both the Schur elements and the characters $\chi_\l(b_w)$ are positive, and we conclude with the explicit expression of \Cref{Thm:explexpr-boundary}.
	\end{proof}
	
	For all other types the characters $\chi_\l(b_w)$ can have negative coefficients and hence the invariant can have negative coefficients as well. It is for example the case for $g=2, k=1, W=B_2$ and $\alpha=b_{rsr}$.
	
	Using the link between the standard trace and the expansion of an element in a given basis, we get:
	\begin{coro}
		In type $A$, any expression of the form $(C_wC^w)^{k}(C_xC_yC^xC^y)^{g}$ has positive coefficients in the dual Kazhdan--Lusztig basis $(b^w)$.
	\end{coro}
	\begin{proof}
		The coefficient of $g = (C_wC^w)^{h}(C_xC_yC^xC^y)^{g}$ along $b^w$ when expressed in the KL basis is $\tr b_w g$ which is positive by the previous corollary.
	\end{proof}
	
	\begin{example}
		For $W = \mathfrak{S}_2$, the Hecke algebra is commutative and the invariant is given by $\tr ((C_1C^1+C_sC^s)^m)$ in any basis $(C_1,C_s)$. The dual KL basis is $(b^1, b^s)=(h_1-q^{-1/2}h_s, h_s)$. Using induction, one easily checks that
		\begin{equation}
		\tr ((b_1b^1+b_sb^s)^m) = ((1+q)^{m-1}+(1+q^{-1})^{m-1})b^1 + q^{(m-1)/2}(q^{1/2}+q^{-1/2})^mb^s~.
		\end{equation}
		Each coefficient is indeed positive. Note that the first coefficient is $P_{0,m+1,\mathfrak{S}_2}(q)$.
	\end{example}

	If we use labels with positive coefficients in the basis adapted to the Wedderburn decomposition, we get:
	\begin{prop}
		Consider a ciliated surface with labeled boundary, such that the product of the labels on each boundary component has positive coefficients in the Wedderburn-adapted basis. Then the polynomial invariant is positive.
	\end{prop}
	\begin{proof}
		This follows from the explicit expression in \Cref{Thm:explexpr-boundary} and the fact that the characters evaluated at the Wedderburn adapted basis are non-negative.
	\end{proof}

	\section{Perspectives}
	
	Our construction can be generalized in different seemingly promising ways, that we list and comment briefly.
	
	\paragraph{Generalization to other symmetric algebras.}	
	Our construction of a $2$-dimensional quantum field theory for which cobordisms are ciliated surfaces seems to work for any symmetric finitely generated algebra over \textit{nice} rings, for example for cyclotomic Hecke algebras and Yokonuma--Hecke algebras \cite{Maria-HDR}.  We believe that the explicit expressions of \ref{Thm:explexpr} and \ref{Thm:explexpr-boundary} stay true, but properties like the invariance under $q\mapsto q^{-1}$, positivity or the interpretation via counting of higher laminations might not hold in general. 
	
	\paragraph{Generalization to affine Hecke algebras.}	
	As already underlined in the introduction and in \Cref{Rem:higherlaminations}, our original motivation for the present work was to study the space of functions over character varieties and its canonical basis: higher laminations. The Satake correspondence is expected to play a role in this story. Since it identifies the spherical affine Hecke algebra corresponding to an algebraic reductive Lie group $G$ with the space of representations of the Langlands dual $G^\vee$, the generalization of the TQFTs constructed above to spherical affine Hecke algebras should be related to the corresponding character varieties. Since those symmetric algebras are not finitely generated anymore, the construction of the TQFTs must involve some regularization of the infinite sums which then appear in the gluing process. Higher laminations, which would generalize finite higher laminations to this new setup at least in the case of affine Hecke algebras, would be very similar to the spectral networks of \cite{gaiotto2013wall}, calling for a physical interpretation of our construction and possibly an understanding in terms of BPS states counting in $4d$ $\mathcal{N}=2$ theories of class $S$.
	
	\paragraph{Categorification.}	
	Hecke algebras are categorified by Soergel bimodules \cite{soergel2007kazhdan}. Can our construction also be expressed in term of these bimodules? Whereas it was tempting to look for such a categorified version of our TQFTs to explain the positivity properties of the invariants we observed at the very beginning of our study, we have seen that the positivity does not hold in all cases. We still expect a possible categorification - maybe only for type $A$ - which would explain the positivity properties of the Schur elements of the Hecke algebras as described in \Cref{Appendix:Schur-positive}, or the positivity properties of the invariants corresponsing to punctured surfaces.

	\appendix
	
	\section{Computation using Sage and CHEVIE}\label{Appendix:compute}
	In this appendix, we describe two ways to compute our invariants by computer, the first using Sage \cite{sage} and the second using CHEVIE, a package of Gap3 (see \cite{michel2015development} and \cite{geck1996chevie}).
	
	\subsection{Sage}
	We recommend the online platform \href{https://sagecell.sagemath.org/}{SageMathCell} where you can perform computations using Sage without any installation.
	To compute the polynomial invariant, we use Equation \eqref{easyexpr}. For example, a possible code for computing the polynomial $P_{1,3,A_3}$ is:
	
	\begin{verbatim}
	k = 3
	g = 1
	R.<v> = LaurentPolynomialRing(QQ)
	H = IwahoriHeckeAlgebra('A3', v, -1/v)
	W = H.coxeter_group()
	T = H.T(); 
	S = (sum(T(i)*T(i.inverse()) for i in W))**(k-1)*(sum
	(sum(T(i)*T(j)*T(i.inverse())*T(j.inverse()) for j in W)
	for i in W))**g 
	\end{verbatim}
	
	Using the parameters $v$ and $-1/v$ for H corresponds to the normalized version of the Hecke algebra in which $(h_s+v)(h_s-1/v)=0$. An H.T(w) for w in the Coxeter group is the standard basis element of H corresponding to w, while H.Cp(w) is the element in the KL basis associated with w.
	
	This code yields the sum $(h_wh^w)^{k-1}(h_xh_yh^xh^y)^{g}$. However we are interested in its trace, hence the $P_{g,k,W}$ is the coefficient of $1$ in the result, in which one can replace $v^{-2}$ by $q$.
	Here are some examples of computations:
	
	\begin{itemize}
		\item $P_{0,3,\mf{S}_3}(q) = q^3+2q^2+10q+10+10q^{-1}+2q^{-2}+q^{-3}.$
		\item $P_{1,1,\mf{S}_3}(q) = q^3+2q^2+4q+4+4q^{-1}+2q^{-2}+q^{-3}.$
		\item $P_{0,4,\mf{S}_3}(q) = q^6+4q^5+8q^4+10q^3+24q^2+36q+50+...$
		\item $P_{0,3,\mf{S}_4}(q) = q^6+ 3q^5+5q^4+33q^3+67q^2+108q+142+...$
		\item $P_{0,3,G_2}(q) = q^6 + 2q^5 + 2q^4 + 2q^3 + 2q^2 + 72q-18+...$
	\end{itemize}

	For a surface with one boundary component labeled by $h$ one needs to multiply the big sum by $h$, and expand the result in the standard basis (to read of the constant term).
	Here is an example for type $A_3$ and $h$ an element in the KL-basis:
	\begin{verbatim}
	Cp = H.Cp()
	r,s,t = W.simple_reflections()
	T(Cp(r*s*t*r)*S)
	\end{verbatim}

	\subsection{CHEVIE}
	CHEVIE is a package of Gap3 (not included in Gap4). We recommend the installation from the \href{https://webusers.imj-prg.fr/~jean.michel/gap3/}{webpage of Jean Michel}.
	The advantage of CHEVIE is that it knows the Schur elements and characters. So we can use the explicit expression from \Cref{Thm:explexpr} to compute our polynomial.
	
	Here is a code computing $P_{0,3,E_8}$:
	\begin{verbatim}
	g:=0;;
	k:=3;;
	W:= CoxeterGroup("E",8);;
	v:=X(Cyclotomics);; v.name:="v";;
	H:=Hecke(W,[[v,-v^-1]]);;
	T:=Basis(H,"T");;
	Cp:=Basis(H,"C'");;
	schur:=SchurElements(H);;
	list:=[1..Length(schur)];;
	dim:=HeckeCharValues(T());;
	Sum(list,i->dim[i]^k*schur[i]^(2*g-2+k));
	\end{verbatim}
	
	Using the explicit formula from \Cref{Thm:explexpr}, we get the following general formula for type $A_2$ where $m=2g-2+k$:
	$$P_{g,k,\mf{S}_3}(q) = (1+2q+2q^2+q^3)^m+(1+2q^{-1}+2q^{-2}+q^{-3})^m+2^k(q+1+q^{-1})^m.$$
	For type $G_2$:
	\begin{align*}
	P_{g,k,G_2}(q) =& \; (1+2q+...+2q^5+q^6)^{m}+(1+2q^{-1}+...+2q^{-5}+q^{-6})^{m} \\
	& \; +2(3q^{-1}+6+3q)^{m} + 2^k(6q-6+6q^{-1})^m+2^k(2q+2+2q^{-1})^m.
	\end{align*}
	
	For a surface with one boundary component, we can compute the polynomial using \Cref{Thm:explexpr-boundary}: one only need to change the last row of the code above to: 
	
	\begin{verbatim}
	h:=Cp(1);;
	Sum(list,i->dim[i]^(k-1)*schur[i]^(2*g-2+k)*HeckeCharValues(h)[i]);
	\end{verbatim}
	
	One can of course change the value of $h$ at will.

	\section{Positivity for Schur elements}\label{Appendix:Schur-positive}
	We study in detail the coefficients of Schur elements associated to a Iwahori--Hecke algebra and their positivity. The main tool is an explicit formula for the Schur elements using generalized hook lengths from \cite{Maria-HDR} (see in particular Example 2.5 for more references).
	
	\begin{thm}\label{Thm:Schur-pos}
		The Schur elements $s_\l(q)$ have positive coefficients for all Coxeter groups of classical type and for the exceptional types $E_6$ and $E_7$.
	\end{thm}
	\begin{proof}
		For the exceptional types $E_6$ and $E_7$ an explicit computation using CHEVIE shows the positivity of the Schur elements, while in all other exceptional types there are Schur elements with negative coefficients.
		
		For the classical types we use the explicit formula from Theorem 4.3 in \cite{Maria-HDR} which was first published in \cite{chlouveraki2012schur}.
		We need to introduce some notations to state the formula. For classical types, an irreducible representation $\l$ is described by a set of Young diagrams $(\l^{(0)},...,\l^{(l-1)})$ also called a multipartition (for Weyl groups we have $l=1$ or $l=2$). The \textit{generalized hook length} at $(i,j)\in \l$ with respect to two Young diagrams $\l$ and $\mu$ is
		\begin{equation}
		h_{i,j}^{\l,\mu} = \l_i-i+\mu_j'-j+1~,
		\end{equation}
		where $\l_i$ denotes the length of the $i$-th row in $\l$ and $\mu'$ is the conjugated Young diagram. For $\l=\mu$ this gives the usual hook length.
		
		We can now state the explicit formula. The Schur element indexed by a multipartition $\l=(\l^{(0)},...,\l^{(l-1)})$ of the integer $n$ is given by
		\begin{equation}
		s_\l(q) = (-1)^{n(l-1)}q^{-N(\bar{\l})} \prod_{0\leq s\leq l-1}\prod_{(i,j)\in\l^{(s)}}\left([h_{i,j}^{\l^{(s)},\l^{(s)}}]_q \prod_{0\leq t\leq l-1, t\neq s}(q^{h_{i,j}^{\l^{(s)},\l^{(t)}}}Q_sQ_t^{-1}-1) \right)~.
		\end{equation}
		
		Note that we have not introduced the notation $N(\bar{\l})$ because we can ignore this part of the formula for our purposes.
		
		Type $A$ corresponds to $l=1$ hence the formula simplifies to:
		\begin{equation}\label{schur-an}
		s_\l(q) = q^{-N(\bar{\l})} \prod_{i,j \in \l}[h_{i,j}^{\l}]_q~.
		\end{equation}
		The Schur element is thus a product of \textit{quantized} hook lengths (modulo a shift). These quantum integers are all positive, and hence is the Schur element $s_\l(q)$.
		
		Type $B$ corresponds to $l=2, Q_0=q$ and $Q_1=-1$. An irreducible representation is parameterized by a pair of Young diagrams $(\l, \mu)$, which yields:
		\begin{equation}
		s_{\l,\mu}(q) = q^{-N(\l\cup\mu)} \prod_{(i,j)\in\l}[h_{i,j}^{\l,\l}]_q (q^{2+\l_i-i+\mu_j'-j}+1)\times\prod_{(i,j)\in\mu}[h_{i,j}^{\mu,\mu}]_q(q^{\mu_i-i+\l_j'-j}+1)~.
		\end{equation}
		Since the quantum integers are positive, the Schur element again has positive coefficients.
		
		To obtain the results for type $D$, we have to use a link to type $B$ which is established in \cite[Part 2.3]{chlouveraki2009blocks} using Clifford theory. The result is that we can reduce type $D$ to $l=2$ with parameters $Q_0=q$ and $Q_1=-q$.
		An irreducible representation is given by an unordered pair of Young diagrams $(\l,\mu)$. If $\l\neq \mu$ we get:
		\begin{equation}
		s_{\l,\mu}(q) = \frac{1}{2} q^{-N(\l\cup\mu)} \prod_{(i,j)\in\l}[h_{i,j}^{\l,\l}]_q (q^{\l_i-i+\mu_j'-j+1}+1)\times\prod_{(i,j)\in\mu}[h_{i,j}^{\mu,\mu}]_q(q^{\mu_i-i+\l_j'-j+1}+1)~.
		\end{equation}
		
		For $\l=\mu$ we get
		\begin{equation}
		s_{\l,\l}(q) = q^{-N(\l\cup\l)} \prod_{(i,j)\in\l}[h_{i,j}^{\l,\l}]_q^2 (q^{\l_i-i+\mu_j'-j+1}+1)^2~.
		\end{equation}
		
		In both cases the expressions show the positivity of the Schur elements.
	\end{proof}

	Apart from the positivity, the coefficients of the Schur elements for Coxeter groups satisfy other interesting properties. Call a sequence of integers $(a_1, a_2, ..., a_n)$ \textit{symmetric} if $a_i = a_{n-i}$ and call it \textit{log-concave} if $a_i^2 \geq a_{i-1}a_{i+1}$ (taking the logarithm gives precisely the condition for a concave function).
	
	\begin{prop}\label{Prop:schur-pos-sym}
		For any Coxeter group, the coefficients of the Schur elements are symmetric.
		In type $A_n$, the coefficients are in addition log-concave.
	\end{prop}
	The proof of the first part relies on a general formula for Schur elements in terms of cyclotomic polynomials which are symmetric.
	The second part relies on Chlouveraki--Jacon's formula and the observation that log-concavity is multiplicative.
	\begin{proof}
		In \cite{Maria-HDR}, Formula (2.2) gives the following expression of Schur elements:
		\begin{equation}
		s_\l = \xi_\l q^{-a_\l}\prod_{\Phi\in \Cyc_\l}\Phi(q^{n_{\l,\Phi}})~,
		\end{equation}
		where $\xi_\l \in \R$, $n_{\l,\Phi}\in\Z_{>0}$ and $\Cyc_\l$ is a set of cyclotomic polynomials. This formula comes from a case-by-case study.
		
		Since we are interested in the symmetry of the sequence of coefficients, we can ignore the prefactor $\xi_\l q^{-a_\l}$. Furthermore, cyclotomic polynomials are known to be symmetric. Finally a product of symmetric polynomials is still symmetric. This proves the first part.
		
		In type $A_n$, we have seen above in Equation \eqref{schur-an} that 
		\begin{equation}
		s_\l(q) = q^{-N(\bar{\l})} \prod_{i,j \in \l}[h_{i,j}^{\l}]_q~.
		\end{equation}
		The sequence of coefficients in a quantum integer is $(1,1,...,1)$ which is positive and log-concave. By Proposition 2 in \cite{stanley1989log}, the product of positive log-concave sequences preserves these properties, thus the coefficients of the Schur elements are log-concave.
	\end{proof}

	\bibliographystyle{alpha}
	\bibliography{ref}
	
\end{document}